\documentclass[10pt]{amsart}
\usepackage{amsfonts, amsmath, enumerate, amssymb}

\newtheorem{theorem}{Theorem}[section]

\newtheorem{corollary}[theorem]{Corollary}
\newtheorem{proposition}[theorem]{Proposition}
\theoremstyle{definition}
\newtheorem{definition}[theorem]{Definition}
\theoremstyle{notation}

\theoremstyle{remark}
\newtheorem{remark}[theorem]{Remark}
\newtheorem{remarks}[theorem]{Remarks}

\numberwithin{equation}{section}
\usepackage{amscd}
\usepackage{xypic}
\usepackage{amsmath, amssymb}

\numberwithin{equation}{subsection}
\newcommand{\be}%
  {\protect\setcounter{equation}{\value{subsubsection}}}
  \newcommand{\ee}%
   {\protect\setcounter{subsubsection}{\value{equation}}}

  {\protect\setcounter{subsubsection}{\value{equation}}}

\def \rma{\rm a}

\def \BU{\rm {BU}}

\def \BGL{\rm {BGL}}
\def \BN{\rm {BN}}
\def \B{\rm B}
\def \rmB{\rm B}
\def \BG{\rm {BG}}

\def \Cl{\mathbb C}
\def \colim{\underset \rightarrow  {\hbox {lim}}}
\def \colimm{\underset {m \rightarrow \infty}  {\hbox {lim}}}
\def \colimi{\underset {i \rightarrow \infty}  {\hbox {lim}}}
\def \colimn{\underset {n \rightarrow \infty}  {\hbox {lim}}}

\def \colimK.{\underset {\underset K^.  \rightarrow}  {\hbox {lim}}}

\def \colimU.{\underset {\underset U_.  \rightarrow}  {\hbox {lim}}}

\def \colimE{\underset {\underset E  \rightarrow}  {\hbox \be \begin{equation}
			\xymatrix{ {[\BGL _{\infty}, \rmQ (\BN (\rmT ))]}  \ar@<ex>[r]^{p_*} \ar@<1ex>[d]^{p^*} &{\tilde K^0(\BGL _{\infty})} \ar@<1ex>[r]^c \ar@<1ex>[d]^{p^*} & {h^0(\BGL _{\infty})} \ar@<1ex>[d]^{p^*} \\
				{[\rmQ (\BN (\rmT ) ), \rmQ (\BN (\rmT ) )]}  \ar@<1ex>[r]^{p_*}  &{\tilde K^0(\BN (\rmT ) )} \ar@<1ex>[r]^c  & {h^0(\BN (\rmT ) ) } }
		\end{equation} \ee{lim}}}

\def \rmd{\rm d}

\def \rmE{{\rm E}}
\def \cE{\mathcal E}

\def \EG1{E{(G \times {\mathbb C}^*)}{\underset {G\times {\mathbb C}^*} 
\times}}

\def \EZ(s)1{E{(Z(s) \times {\mathbb C}^*)}{\underset {(Z(s)\times {\mathbb
C}^*)}  \times}}

\newcommand{\eps}{ \, {\boldsymbol\varepsilon} \,}

\def \EM(u){EM(u){\underset {M(u)}  \times}}
\def \EM(us){EM(u,s){\underset {M(u, s)}  \times}}

\def \EG{\rm {EG}}
\def \EGL{\rm {EGL}}

\def \GL{\rm GL}
\def \rmG{\rm G}

\def \bH{\rm {\bf H}}

\def \rmH{\rm H}
\def  \rmh{\rm h}

\def \invlim1{\underset {\infty \leftarrow q}  {\hbox {lim}}^1}

\def \rmK{\rm K}

\def \L3{\Lambda \times \Lambda \times \Lambda}
\def \L2{\Lambda \times \Lambda}

\def \lim{\underset \leftarrow  {\hbox {lim}}}

\def \longright2arrow{{\overset \longrightarrow  {\overset {} 
\longrightarrow}}}

\def \L{L\times \Cl ^*}

\def \rmN{\rm N}

\def \P{\mathbb P}
\def \rmP{\rm P}

\def \Spt{\rm {Spt}}

\def \rmp{\rm p}

\def \rmQ{\rm Q}

\def \rmq{\rm q}

\def \ra{\rightarrow}

\def \RG^{R(G)^{\hat {}}\ }

\def \res{respectively}
\def \rmr{\rm r}

\def \SH{{\mathcal S}{\mathcal H}}
\def \Sm{\rm {Sm}}
\def \rmS{\rm S}

\def\Spt{\rm {\bf Spt}}

\def \SH{{\mathcal S}{\mathcal H}}
\def\Spt{\rm {\bf Spt}}

\def \St{\rm St}

\def \topGcoh*{^{top, *} _{G}}
\def \topGho*{ _{top,*} ^{G}}

\def \T{{\mathbf T}}

\def \rmT{\rm T}

\def \Th{\rm Th}

\def \rmU{\rm U}

\def \rmV{\rm V}

\def \rmW{\rm W}

\def \X{\mathcal X}

\def \rmX{\rm X}

\def \Y{\mathcal Y}
\def \rmY{\rm Y}

\def \Z(s){Z(s) \times {\mathbb C}^*}
\def \Z{\mathcal Z}

\def \rmZ{\rm Z}

\begin{document}

\title{The Motivic Segal-Becker Theorem for Algebraic K-Theory}

\author{Roy Joshua}
\address{Department of Mathematics, Ohio State University, Columbus, Ohio,
	43210, USA.}
\email{joshua.1@math.osu.edu}
\author{Pablo Pelaez}
\address{Instituto de Matem\'aticas, Ciudad Universitaria, UNAM, DF 04510, M\'exico.}
\email{pablo.pelaez@im.unam.mx}
\thanks{  }  
\begin{abstract}  The present paper is a continuation of earlier work by Gunnar Carlsson and the first author on a motivic variant of the classical Becker-Gottlieb transfer and an additivity theorem for such a transfer by the present authors. 
Here, we establish a motivic variant of the classical Segal-Becker theorem relating the classifying space of a $1$-dimensional torus with the spectrum defining algebraic K-theory. 
\end{abstract}
\maketitle

\centerline{\bf Table of contents}
\vskip .2cm 
1. Introduction
\vskip .2cm
2.  The motivic Segal-Becker theorem
\vskip .2cm
3. The motivic transfer and the motivic Gysin maps associated to projective smooth morphisms
\vskip .2cm  \indent \indent
\markboth{Roy Joshua and Pablo Pelaez}{The Motivic Segal-Becker Theorem}
\input xypic
\vfill \eject

\section{\bf Introduction}
\label{intro}
A classical result due to Segal from the early 1970s (see \cite{Seg})
is a theorem that shows the classifying space of the infinite unitary group,
namely $\BU$, is a split summand of $\colimm \Omega_{S^1}^m((S^1)^m \wedge {\mathbb C}{\mathbb P}^{\infty})$. A year later, Becker (see \cite{Beck})
 proved a similar result for the infinite orthogonal group in the place
 of the infinite unitary group $\rmU$ and ${\rm {BO(2)}}$ in the place of ${\mathbb C}{\mathbb P}^{\infty}$.
 \vskip .2cm
 The purpose of this paper is to consider similar problems in the motivic
 world and for Algebraic K-Theory, making use of a theory of the Motivic Becker-Gottlieb transfer 
 worked out by Gunnar Carlsson and the first author in \cite{CJ20} and the Additivity Theorem for such a transfer worked out in \cite{JP20} by the authors.
 We will adopt the terminology and conventions from \cite{CJ20} as well as
 other terminology that has now become standard. 
 As such, the base scheme will be a perfect field ${\it k}$ 
 and we will restrict to the category of smooth schemes of finite type over ${\it k}$.
 This category will be denoted $\Sm({\it k})$ and will be provided with the {\it Nisnevich topology}. ${\rm {PSh}}_*({\it k})$ will denote the category of pointed simplicial presheaves 
 on this site. This category will be made motivic, by inverting the affine line ${\mathbb A}^1$ as in \cite{MV}: the pointed simplicial presheaves in this category will be referred to as motivic spaces.  We will let $\T={\mathbb P}^1$ pointed at $\infty$. We will denote $\T^{\wedge n}$ throughout by $\T^n$. Then a motivic spectrum
 $\rmE$ will denote a sequence $\{\rmE_n|n \ge 0\}$ of motivic spaces provided with structure maps $\T \wedge \rmE_n \ra \rmE_{n+1}$. The category of motivic spectra will be denoted $\Spt({\it k})$, or just $\Spt$ if the choice of ${\it k}$ is clear. Then a motivic spectrum $\rmE$ will be called an $\Omega_{\T}$-spectrum if it is level-wise fibrant and the adjoint to the structure maps given by $\{\rmE_n \ra \Omega_{\T}(\rmE_{n+1})|n\}$ are all motivic weak-equivalences. 
 \vskip .2cm
Then, the first observation (see \cite[6.2]{VV98}) is that Algebraic K-theory is represented by the
motivic-spectrum with ${\mathbb Z} \times \BGL_{\infty}$ as the motivic space in each degree, with the structure map given by the Bott-periodicity:
${\mathbb Z} \times \BGL _{\infty} \simeq \Omega_{\T}({\mathbb Z} \times \BGL _{\infty})$. We will denote this motivic spectrum by ${\mathbf K}$. Therefore,
\[\rmK ^0(X) \simeq [\Sigma_{\T}^{\infty}(X_+), {\mathbf K}] \cong [X, {\mathbb Z} \times \BGL _{\infty}] \]
where the first $[\quad, \quad]$ (the second $[\quad, \quad]$) denotes
the hom in the stable motivic homotopy category (the corresponding unstable pointed motivic  homotopy category, \res.)
\vskip .2cm
We observe in Proposition ~\ref{BGL.Omega} that there is an $\Omega_{\T}$-motivic spectrum whose $0$-th term is given by the motivic space $\BGL _{\infty}$. Assuming this, 
 the first main result of this paper is the following theorem, which we call the {\it motivic
	Segal-Becker Theorem} in view of the fact that such a result was proven for topological complex K-theory, making use of  complex unitary groups, by Segal (see \cite{Seg}) 
and for real K-theory, making use of orthogonal groups, by Becker (see \cite{Beck}). (In fact, Becker's proof, making use of the transfer, also applies to topological complex K-theory.) For a motivic space $\rmP$, we will let $\rmQ ( \rmP ) = \colimn \Omega_{\T}^n \T^{\wedge n}(\rmP)$.
Of key importance for us is the following map:
\be \begin{equation}
  \label{splitting.map}
\lambda: \rmQ(\B{\mathbb G}_m) \ra \rmQ(\colim_n \BGL _{n}) = \rmQ(\BGL _{\infty}) {\overset q \ra} \BGL _{\infty},
\end{equation} \ee
\vskip .1cm \noindent
where the map $q$ is the obvious one induced by the fact that $\BGL_{\infty}$ is the $0$-th space of an $\Omega_{\T}$-spectrum: see Proposition ~\ref{BGL.Omega}. The map $\rmQ(\B{\mathbb G}_m) \ra \rmQ(\colim_n \BGL _{n}) = \rmQ(\BGL _{\infty})$ 
is induced by the inclusion,
${\mathbb G}_{\rm m} \ra \GL _n \ra \GL _{\infty}$, where the first map is the diagonal imbedding.
\begin{theorem} 
	\label{mainthm.1}
	(The motivic Segal-Becker theorem for Algebraic K-Theory)
	(i) Assume that the base scheme is a field ${\it k}$ of characteristic $0$. Then the map in ~\eqref{splitting.map}  induces a  surjection for every pointed motivic space $\rmX$ that is a compact object
	in the  unstable pointed motivic  homotopy category:
	\[ [ \rmX, \rmQ (\B{\mathbb G}_m)] \ra [\rmX, \BGL _{\infty}].\]
	(Recall that a motivic space $\rmX$ is a compact object in the unstable pointed motivic homotopy category, if $Map(X, \quad)$
	commutes with all small colimits in the second argument, and where $Map(\quad, \quad)$ denotes the simplicial mapping space.)
	\vskip .1cm \noindent
	(ii) Assume that the base scheme is a perfect field ${\it k}$  of positive characteristic $p>0$. \footnote{The assumption that
	${\it k}$ be perfect may be dropped in view of recent results such as in \cite{EK} and \cite[Theorem 10.12]{BH}.}
	Then, after inverting $p$,  the map in ~\eqref{splitting.map}  
	induces a surjection for every pointed motivic space $\rmX$ that is a compact object
	in the corresponding unstable pointed motivic homotopy category:
	\[ [ \rmX, \rmQ (\B{\mathbb G}_m)] \ra [\rmX, \BGL _{\infty}].\]
\end{theorem} 
\vskip .2cm
\begin{remarks} 1. Localizing at the prime $p$ in the unstable pointed motivic homotopy category, as used in statement (ii) and elsewhere in this paper
is discussed in detail in \cite{AFH}. One may also observe that, though $[\quad, \quad]$ as used in statement (ii) denotes  
Hom in the unstable pointed motivic  homotopy theory, since the target space is an infinite $\T$-loop space, the above Hom identifies readily with
a Hom in the motivic stable homotopy category, after making use of the adjunction between taking $\Omega_{\T}$-loops and $\T$-suspension.
 \vskip .1cm \noindent
2. One should view the above results as a rather weak-form of the Segal-Becker theorem, in the sense that
 we are able to prove only the surjectivity and (not split surjectivity) of the above maps, and also only for objects
  $\rmX$ that are compact objects in the corresponding  unstable pointed motivic homotopy category. We hope to consider questions on
  split surjectivity in a sequel to this paper, as it seems to involve considerable additional work and certain techniques used in
  establishing such splittings classically do not seem to extend readily to the motivic framework.
  \vskip .1cm \noindent
  3. It is possible there is an analogue of the above theorem for Hermitian K-theory (see: \cite{Ho05}, \cite{Sch16}) which is represented by the classifying space
 of the infinite orthogonal groups. In fact, much of the proof for the case of algebraic K-theory seems to carry over to the Hermitian case, the main difficulty
 being to prove an analogue of Theorem ~\ref{compat.Gysin}. We hope to return to this question elsewhere.
\end{remarks}
\vskip .2cm
Our approach to all of the above is via a theory of motivic transfers. Such 
a theory of motivic and \'etale variants of the classical Becker-Gottlieb transfer were developed by Gunnar
Carlsson  and the first author in \cite{CJ20} and the additivity of transfers was established in a general framework 
by the present authors in
\cite{JP20}, though special cases such as Snaith-splitting for the suspension spectrum of ${\rm BGL_n}$ appears in \cite{K18}.
Theorem ~\ref{mainthm.1} is proven by making intrinsic use of this transfer, just as was done by Becker and Gottlieb, making use of the classical Becker-Gottlieb transfer. 
See  \cite{Beck}, (and also, \cite{BG75} and \cite{BG76}). 
\vskip .2cm
{\it In fact we summarize the main ideas of the proof of Theorem ~\ref{mainthm.1}
(as well as an overview of the paper) as follows}: The splitting provided by the motivic Becker-Gottlieb transfer as in Proposition ~\ref{split.0} enables us to prove
Proposition ~\ref{key.1}. This shows the map 
\[\bar q= q \circ \rmQ(\rmp): \rmQ ( \BN_{\rm GL_{\infty}} (\rmT ) )= \rmQ (\colim_n \BN_{\rm GL_n}(T_n) ) \ra \rmQ (\colim_n \BGL _{n}) = \rmQ (\BGL _{\infty}) {\overset q \ra } \BGL _{\infty}\]
induces a surjection
\[ [\rmX,  \rmQ ( \BN_{\rm GL_{\infty}} (\rmT ) )] {\overset {{\bar q}_*} \ra} [\rmX, \BGL _{\infty}],\]
for every compact object $\rmX$ in the unstable pointed motivic  homotopy category. (Here ${\rm N}_{\rm GL_n}({\rm T}_n)$ denotes the normalizer of 
the maximal torus of diagonal matrices in ${\rm GL}_n$.) Then we show in Propositions ~\ref{key.comm.triangle.0} and ~\ref{key.diagms} that the map $\bar q_*$ above factors through $\lambda_*$,
where $\lambda$ is the map in ~\eqref{splitting.map}, thereby proving the theorem. 
\vskip .2cm
It may be worth pointing out
this involves {\it a second, somewhat different use of the transfer, this time as defined in ~\eqref{rel.transf}, and with
a key property proven in Corollary ~\ref{tr.finite.etale.maps}.}
These occupy most of section 2 of the paper. 
While Proposition ~\ref{key.1} is rather
straightforward given the properties of the motivic Becker-Gottlieb transfer, Proposition ~\ref{key.diagms} is a bit involved:
here one needs to know the relationship between maps defined by the transfer as in ~\eqref{rel.transf} and Gysin maps, for at least finite \'etale maps in orientable motivic cohomology theories.  This is
 discussed in section 3 of the paper.
\vskip .2cm \noindent
{\bf Acknowledgment.} The authors would like to thank Gunnar Carlsson for 
helpful discussions and to a referee for undertaking a careful and detailed reading of the paper and for making several critical
comments and valuable suggestions, that helped the authors to sharpen their
results and improve the exposition. Needless to say the authors are greatly indebted to James Becker whose proof
of the corresponding result in topology (see \cite{Beck}) 
serves as the basis of our current work.

\vskip .2cm
\subsection{Basic assumptions and terminology}
We will assume throughout that the base scheme is a perfect field. (The assumption that
	${\it k}$ be perfect can be dropped, if one prefers, in view of recent results such as in \cite{EK} and \cite[Theorem 10.12]{BH}.) Then $\Spt = \Spt({\it k})$ will
denote the  category of motivic spectra on the big Nisnevich site of ${\it k}$, with $\SH = \SH({\it k})$ denoting the
corresponding motivic stable homotopy category. If ${\it k}$ is of characteristic $0$, no further assumptions are needed.
\vskip .1cm
However, if $char ({\it k}) = p>0$, then we  will
only consider $\SH[p^{-1}]$, which is the motivic stable homotopy category on ${\it k}$, with the prime $p$ inverted. (The main reason for this restriction is that a theory of
Spanier-Whitehead duality holds only after inverting $p$ in this case.) Given a motivic spectrum ${\rm E}$, and a motivic space $\rmX$,
the generalized motivic cohomology represented by ${\rm E}$ is given by the bi-graded theory
\be \begin{equation}
  \label{gen.motivic.coh}
  {\rm h}^{\rm p, q}(\rm X, {\rm E}) = [\Sigma_{\T}^{\infty}\rmX, ({\rm S}^1)^{p-q} \wedge {\mathbb G}_m^{\wedge q} {\rm E}]
    \end{equation} \ee
with $[\quad, \quad ]$ denoting the Hom in the motivic stable homotopy category.

\vskip .1cm 
In both the above cases, we do not require the existence of a symmetric monoidal structure
on the category of spectra itself, that is, it is sufficient to assume the smash product of spectra is homotopy 
associative and homotopy commutative. 
\vskip .1cm
\subsection{\bf Geometric classifying spaces}
\label{geom.class.sp}
We begin by recalling
briefly the construction of the {\it geometric classifying space of a linear algebraic group}: see for example, \cite[section 1]{Tot}, \cite[section 4]{MV}. Let 
$\rmG$ denote a linear algebraic group over $\rmS =Spec \, {{\it k}}$, that is,  a closed subgroup-scheme in $\GL_n$ over $\rmS$ for some n. For a  (closed) imbedding 
$i : \rmG \ra \GL_n$ as a closed subgroup-scheme, {\it the geometric classifying space} $\rmB_{gm}(\rmG; i)$ of $\rmG$ with respect to $i$ is defined as follows. For $m \ge   1$, let 
$\rmE\rmG^{gm,m}=U_m(G)=U({\mathbb A}^{nm})$ be 
the open sub-scheme of ${\mathbb A}^{nm}$ where the diagonal action of 
$\rmG$ determined by $i$ is free. By choosing $m$ large enough, one can always ensure that 
$\rmU({\mathbb A}^{nm})$ is non-empty and the quotient $\rmU({\mathbb A}^{nm})/G$ is a quasi-projective scheme.
We will further choose such a family $\{\rmU({\mathbb A}^{nm})|m\}$ so that it satisfies the hypotheses in \cite[Definition 2.1, section 4.2]{MV} defining an
 {\it admissible gadget}.
\vskip .1cm
Let $\rmB\rmG^{gm,m}=\rmV_m(\rmG)=U_m(\rmG)/\rmG$ denote the quotient 
$\rmS$-scheme (which will be a quasi-projective variety ) for the 
action of $\rmG$ on $\rmU_m(\rmG)$ induced by this (diagonal) action of $\rmG$ on ${\mathbb A}^{\rm nm}$; the projection $ \rmU_m(\rmG) \ra \rmV_m(\rmG)$ defines $\rmV_m(\rmG)$ as the 
quotient  of $\rmU_m(\rmG)$ by the free action of $\rmG$ and $\rmV_m(\rmG)$ is thus smooth. We have closed imbeddings 
$\rmU_m(\rmG) \ra \rmU_{m+1}(\rmG)$ and $\rmV_m(\rmG) \ra \rmV_{m+1}(\rmG)$ corresponding to the imbeddings 
$Id \times  \{\rm0\} : {\mathbb A}^{nm} \ra {\mathbb A}^{nm } \times {\mathbb A}^n$. We set $\rmE\rmG^{gm} = \{U_m(\rmG)|m\} = \{ \rmE\rmG^{gm,m}|m\}$ and 
$ \rmB\rmG^{gm} = \{ \rmV_m(\rmG)|m\}$ which are ind-objects in the category of schemes. (If one prefers, one may view each $\rmE\rmG^{gm,m}$ ($\rmB\rmG^{gm,m}$)
as a sheaf on the big Nisnevich (\'etale) site of smooth schemes over ${\it k}$ and then view $\rmE\rmG^{gm}$ ($\rmB\rmG^{gm}$) as the 
the corresponding colimit taken in the category of sheaves on $({\rm {Sm/{\it k}}})_{Nis}$ or on $({\rm {Sm/{\it k}}})_{et}$.)
\vskip .1cm
\begin{definition}
 \label{BG}
  We will denote $\rmE\rmG^{gm}$ by ${\rm EG}$ and $\rmB\rmG^{gm}$ by ${\rm BG}$ throughout the paper.
\end{definition}
\vskip .1cm
Given a scheme $\rmX$  of finite type over $\rmS$ with a $\rmG$-action, we let $\rmU_m(\rmG){\underset {\rmG} \times} \rmX$ denote the {\it balanced product}, 
where $(u, x)$ and $(ug^{-1}, gx)$ are identified for all $(u, x) \eps \rmU_m \times \rmX$ and $g \eps \rmG$.  Since the $\rmG$-action on $\rmU_m(\rmG)$ is free, $\rmU_m(\rmG){\underset {\rmG} \times} \rmX$
exists as a geometric quotient which is also a quasi-projective scheme in this setting, in case $\rmX$ is assumed to be quasi-projective: see \cite[Proposition 7.1]{MFK}. (In case $\rmX$ is an algebraic space of
finite type over $\rmS$, the above quotient also exists, but as an algebraic space of finite type over $\rmS$.)
\vskip .2cm
Next we recall a particularly nice way to construct geometric classifying spaces for closed subgroups of ${\rm GL_n}$ making use of
the Stiefel varieties.
\vskip .2cm
\begin{definition}
\label{BGL.1}
(Stiefel varieties and Grassmannians). Let $n$ denote a fixed positive integer and let $i\ge 0$ denote an integer.
 We let $\St_{n+i,n}$ denote the set of all $(n+i)\times n$-matrices of rank $n$, or equivalently the set of all injective linear transformations
 ${\mathbb A}^n \ra {\mathbb A}^{n+i}$. We view this as an open subscheme of the affine space ${\mathbb A}^{(n+i)\times n}$. The group
 ${\rm GL}_n$ acts on $\St_{n+i, n}$ through its action on ${\mathbb A}^n$: we view this as a right action on the 
 set of all $(n+i)\times n$-matrices. This is a free action and the quotient is the Grassmann variety of $n$-planes in ${\mathbb A}^{n+i}$, and denoted
 ${\rm Grass}_{n+i, n}$. 
\end{definition}
\vskip .1cm
As observed on \cite[p. 138]{MV}, for each fixed positive integer $n$, the family $\{\St_{n+i. n}|i \ge 0\}$ satisfies the
 conditions in \cite[Definition 2.1, p. 133]{MV}, so that it defines what is there called an {\it admissible gadget}. Thus 
 $\{\St_{n+i, n}/{\rm H}|i \ge 0\}$ forms finite dimensional approximations to the classifying space for any closed subgroup ${\rm H}$ 
 of ${\rm GL}_n$. Therefore, we will make the following definitions.
 \begin{definition}
 \label{BGL.2}
  \begin{enumerate}[\rm(i)]
\item ${\rm BH}^{\rm gm, i}= \St_{n+i,n}/{\rm H}$, ${\rm BH} = \colimi {\rm BH}^{\rm gm, i}$, and
\item ${\rm BGL}_{\infty} = \colimi \colimn {\St_{n+i, n}}/{\rm GL}_n = \colimn \St_{2n,n}/{\rm GL}_n$.
\end{enumerate}
\end{definition}

For any linear algebraic group $\rmG$, we will let ${\rm BG}$ denote the geometric classifying space defined
above (as in Definition ~\ref{BG} or equivalently in Definition ~\ref{BGL.2}) as an ind-scheme. (We may view this as a motivic space.)

\section{\bf The motivic Segal-Becker Theorem: proof of Theorem ~\ref{mainthm.1}.}
We begin with the following observation due to Voevodsky.
\begin{proposition}
  \label{BGL.Omega}
  There exists a motivic $\Omega_{\T}$-spectrum, $\tilde {\mathbf K}$, whose $0$-th space is 
  given by $\BGL _{\infty}$.
\end{proposition}
\begin{proof} The required spectrum is just $f_1({\mathbf K})$, where ${\mathbf K}$ denotes the $\Omega_{\T}$-spectrum representing algebraic K-theory:
see \cite[Theorem 2.2]{VV}.
That this is the case follows from \cite[Lemma 2.2]{VV01} (which holds unconditionally over any field by \cite[Theorem 7.5.1]{Lev}) and \cite[Theorem 4.1, Lemma 4.6 and its proof]{VV01}.
\end{proof}	 
\vskip .1cm
\subsection{\bf Changing base points}
\label{change.base.pts}
Recall motivic spaces are assumed to be pointed simplicial presheaves. However,
it is often necessary for us to consider a motivic space $\rmY$ viewed as
an unpointed simplicial presheaf and then provide it with an extra base point $+$.
A typical example, we run into in this paper, is when $\rmY$ is the geometric classifying space of a linear algebraic group (denoted $\BG$: see Definition ~\ref{BG}) or a finite degree approximation of it (denoted $\BG^{gm,m}$), both of which are {\it pointed}. However, while considering
a motivic Becker-Gottlieb transfer involving $\BG$ ($\BG^{gm,m}$), one needs to
consider $\Sigma_{\T}^{\infty}\BG_+$ ($\Sigma_{\T}^{\infty}\BG^{gm,m}_+$), which is the $\T$-suspension spectrum of  $\BG$ ($\BG^{gm,m}$, \res) provided with an extra base point $+$. 
\vskip .1cm
Observe that there is a natural map $\rmr:\BG_+ \ra \BG$ sending
$+$ to the base point of $\BG$. Let $\rma: \Sigma_{\T}^{\infty}\BG \ra \Sigma_{\T}^{\infty}\BG_+$ denote a map, so that $\Sigma_{\T}^{\infty}\rmr \circ \rma =id_{\Sigma_{\T}^{\infty}\BG}$. (Since the definition of such a map $\rma$ is straightforward, we skip the details.)
\begin{proposition}
	\label{split.0}
	Let $\rmh^{*, \bullet}$ denote a generalized motivic cohomology theory defined
	with respect to a motivic spectrum (with $p$ inverted, if $char({\it k}) =p>0$.) Let $\rmG$ denote a linear algebraic group, which is also {\it special in the sense of Grothendieck} (see: \cite{Ch}). Then,  with $\rmN(\rmT)$ denoting the normalizer of a split maximal torus in $\rmG$, one obtains the commutative square
	\be \begin{equation}
	\label{split.diag.0}
	\xymatrix{{\rmh^{*, \bullet}(\Sigma_{\T}^{\infty}\BN(T))  }\ar@<1ex>[r]^{\rmr^*} &{\rmh^{*, \bullet}(\Sigma_{\T}^{\infty}\BN(T)_+)}  \\
		{\rmh^{*, \bullet}(\Sigma_{\T}^{\infty}\BG)  }\ar@<1ex>[r]^{\rmr^*} \ar@<1ex>[u]^{\rmp^*} &{\rmh^{*, \bullet}(\Sigma_{\T}^{\infty}\BG_+)}  \ar@<1ex>[u]^{\rmp^*} ,}  
	\end{equation}\ee
	where the map $\rmp: \BN(T) \ra \BG$ is the map induced by the inclusion
	$\rmN(\rmT) \ra \rmG$. The right vertical map and the horizontal maps are all split monomorphisms. Therefore, the left vertical map is also a monomorphism.
		
\end{proposition}
\begin{proof}
	That the right vertical map is a split monomorphism is a consequence of the
	motivic Becker-Gottlieb transfer as proved in \cite[9.2]{CJ20}, as well as \cite[Theorem 1.6]{JP20} and \cite[Theorem 5.1]{An}.
	Moreover, all of these depend on the key identification of the Grothendieck-Witt group with the motivic $\pi_0$ of the motivic sphere spectrum due to Morel: see
	 \cite{Mo4}, \cite{Mo12}. The restriction that the characteristic of the base field ${\it k}$ be different from $2$ is removed in 
	 \cite[Theorem 10.12]{BH}.
	 \vskip .1cm
	Let the motivic Euler characteristic of $\rmG/\rmN(T)$  be denoted
	$\chi^{{\mathbb A}^1}(\rmG/\rmN(T))$ henceforth. Now one may recall from \cite[Theorem 1.6]{JP20} that we showed 
	$\chi^{{\mathbb A}^1}(\rmG/\rmN(T))$ is $1$ in the Grothendieck-Witt group ${\rm GW}(Spec \,{\it k})$
	(${\rm GW}(Spec \,{\it k})[p^{-1}]$ if $char({\it k})=p>0$), provided
	${\it k}$ has a square root of $-1$. Hence this conclusion holds whenever the base field ${\it k}$ is algebraically or quadratically closed.
	In positive characteristics $p$, one may see that this already shows that $\chi^{{\mathbb A}^1}(\rmG/\rmN(T))$ is a {\it unit}
	in the group ${\rm GW}(Spec \, {\it k})[p^{-1}]$, 
	without the assumption on the existence of a square root of $-1$ in ${\it k}$. For this, one may first observe the commutative diagram, where
	$\bar {\it k}$ is an algebraic closure of ${\it k}$:
	\be \begin{equation}
	     \label{GW.pos.char}
	    \xymatrix{ { {\rm GW}(Spec \, \bar {\it k})[p^{-1}]} \ar@<1ex>[r]^(.6){rk}_(.6){\cong} & {{\mathbb Z}[p^{-1}]}\\
	               { {\rm GW}(Spec \, {\it k})[p^{-1}]} \ar@<1ex>[r]^(.6){rk} \ar@<1ex>[u] \ar@<1ex>[r] & {{\mathbb Z}[p^{-1}]} \ar@<1ex>[u]^{id}.}
	    \end{equation}\ee
Here the left vertical map is induced by the change of base fields from ${\it k}$ to $ \bar {\it k}$, and $rk$ denotes the {\it rank} map. Since the motivic Euler-characteristic
of $\rmG/\rmN(T)$ over $Spec \, {\it k}$ maps to the motivic Euler-characteristic of the corresponding $\rmG/\rmN(T)$ over $Spec \, \bar {\it k}$,
it follows that the rank of $\chi^{{\mathbb A}^1}(\rmG/\rmN(T))$ over $Spec \, {\it k}$ is in fact $1$. By \cite[Lemma 2.9(2)]{An}, this shows that the
 $\chi^{{\mathbb A}^1}(\rmG/\rmN(T))$ over $Spec \, {\it k}$ is in 
 fact a unit in ${\rm GW}(Spec \, {\it k})[p^{-1}]$, that is , when ${\it k}$ has positive characteristic. (For the convenience of the reader, we 
 will summarize a few key facts discussed in \cite[Proof of Lemma 2.9(2)]{An}. It is observed there that when the base field $k$ is {\it not} formally real,
 then ${\rm I}(k) = kernel({\rm GW}(k) {\overset {rk} \ra }{\mathbb Z})$ is the nil radical of ${\rm GW}(k)$: see \cite[Theorem V.8.9, Lemma V.7.7 and Theorem V. 7.8]{Bae}. Therefore,
 if $char(k)=p>0$, and the rank of $\chi^{{\mathbb A}^1}(\rmG/\rmN(T))$ is $1$ in ${{\mathbb Z}[p^{-1}]}$, then $\chi^{{\mathbb A}^1}(\rmG/\rmN(T))$ is $1+q$ for some nilpotent element $q$ in ${\rm I}(k)[p^{-1}]$ and 
 the conclusion follows.)
\vskip .1cm
In characteristic $0$, the commutative
diagram 
\be \begin{equation}
	     \label{GW.char.0}
	    \xymatrix{ { {\rm GW}(Spec \, \bar {\it k})} \ar@<1ex>[r]^(.6){rk}_(.6){\cong} & {{\mathbb Z}}\\
	               { {\rm GW}(Spec \, {\it k})} \ar@<1ex>[r]^(.6){rk} \ar@<1ex>[u] \ar@<1ex>[r] & {{\mathbb Z}} \ar@<1ex>[u]^{id}.}
	    \end{equation} \ee
shows that once again the rank of $\chi^{{\mathbb A}^1}(\rmG/\rmN(T))$ is $1$. Therefore,
 to show that the class $\chi^{{\mathbb A}^1}(\rmG/\rmN(T))$ is a unit in ${\rm GW}(Spec \, {\it k})$, it suffices to show its signature is 
 $1$: this is proven in \cite[Theorem 5.1(1)]{An}. (Again, for the convenience of the reader,
 we summarize some details from the proof of \cite[Theorem 5.1(1)]{An}. When the field $k$ is not formally real, the discussion in
 the last paragraph applies, so that by \cite[Lemma 2.12]{An} one reduces to considering only the case when $k$ is a real closed field. In this case, one lets ${\mathbb R}^{alg}$ denote 
 the real closure of ${\mathbb Q}$ in ${\mathbb R}$. Then, one knows the given real closed field $k$ contains a copy of ${\mathbb R}^{alg}$ and that there exists
 a reductive group scheme $\widetilde \rmG$ over  ${\rm Spec \, } {\mathbb R}^{alg}$ so that $\rmG = \widetilde \rmG \times_{{\rm Spec \, } {\mathbb R}^{alg}} {\rm Spec \, }k$. Let
 $\rm G_{\mathbb R} = \widetilde \rmG \times _{{\rm Spec \, } {\mathbb R}^{alg}} {\rm Spec \, }{\mathbb R}$.  Then 
 one also  observes 
 that the Grothendieck-Witt groups of the three fields $k$, ${\mathbb R}^{alg}$ and ${\mathbb R}$ are isomorphic, and the
 motivic Euler-characteristics $\chi^{{\mathbb A}^1}(\rmG/\rmN(T))$,  $\chi^{{\mathbb A}^1}({\widetilde \rmG}/ {\widetilde {\rm N(T) }})$ and
 $\chi^{{\mathbb A}^1}({\rmG}_{{\rm Spec \, }{\mathbb R}} / { \rmN(T) }_{{\rm Spec \, }{\mathbb R}})$
 over the above three fields identify under the above isomorphisms, so that one may assume 
 the base field $k$ is ${\mathbb R}$. Then it is shown in \cite[Proof of Theorem 5.1(1)]{An} that, in this case,  knowing the rank and signature of the
 motivic Euler characteristic $\chi^{{\mathbb A}^1}(\rmG/\rmN(T))$ are $1$ suffices to 
 prove it is a unit in the Grothendieck-Witt group.)
 \vskip .1cm
 
 These complete the proof that the right vertical map in ~\eqref{split.diag.0} is a split monomorphism.
\vskip .2cm	
	The horizontal maps in ~\eqref{split.diag.0} are split by the map $a^*$. Since the diagram commutes, it follows
	that the left vertical map in ~\eqref{split.diag.0} is also a monomorphism. This proves the Proposition. 
\end{proof}
\begin{remark}
 There is an extension of the above theorem for linear algebraic groups that are not special, as discussed in \cite[Theorem 1.5]{JP20}
 and \cite[Theorem 1.5(1)]{CJ20}. But then we require the field be infinite to prevent certain situations like those mentioned in
 \cite[4.2, Example 2.10]{MV} from occurring. However, for the applications in this paper, we only need to consider the linear
 algebraic groups $\{{\rm GL}_n|n\}$, which are all special.
\end{remark}
\subsection{}
For the rest of the discussion, we will restrict to the family of groups $\{ \GL _n|n\}$. 
Then the main result of this section is the proof of Theorem ~\ref{mainthm.1}.
We will break the proof into several propositions. Given two motivic spaces $\X, \Y$, we will let $[\X, \Y]$ denote the $Hom$ in the unstable pointed motivic homotopy category, with $p$ inverted, if $p>0$ is the characteristic of the base field $k$. For a 
motivic space $\rmP$, we will let $\rmQ (\rmP) = \colimn \Omega_{\T}^n (\T^{\wedge n}(\rmP))$.
\vskip .1cm
Let 
\be \begin{equation}
\label{p}
\rmp:\BN_{\rm GL_\infty} (\rmT) = \colim_n \BN_{\rm GL_n} (T_n) \ra \colim_n \BGL _{n} = \BGL _{\infty}
\end{equation} \ee
denote the map induced by the inclusion of $\rmN_{\rm GL_n}(T_n)$ in ${\rm GL}_n$.
\begin{proposition}
  \label{key.1}
	\begin{enumerate}[\rm(i)]
    \item Assume the base field ${\it k}$ is of characteristic $0$. 
	Then the map 
	\[\bar q= q \circ \rmQ(\rmp): \rmQ ( \BN_{\rm GL_{\infty}}  (\rmT ) )= \rmQ (\colim_n \BN_{\rm GL_n} (T_n) ) \ra \rmQ (\colim_n \BGL _{n}) = \rmQ (\BGL _{\infty}) {\overset q \ra } \BGL _{\infty}\]
induces a surjection for every pointed motivic space $\rmX$ which is a compact object in the unstable pointed motivic homotopy category:
	\[ [ \rmX, \colim_n\rmQ (\BN_{\rm GL_n} (T_n))] \ra [\rmX, \BGL _{\infty}].\]
    \item Assume the base field ${\it k}$ is perfect and of positive characteristic $p$. Then the 
    same conclusion holds after inverting the prime $p$.
       \end{enumerate}
\end{proposition}
\begin{proof} We will follow \cite[\S 4]{Beck} in this proof. The proof of the second  statement follows along the same lines as the proof of the first statement. Therefore,  we will discuss a proof of only the first statement. 
	Clearly the map $\bar \rmq$ provides a  map of the
	corresponding spectra:
\be \begin{equation}
  \label{compos.map}
	 \Sigma_{\T}^{\infty} (\BN_{\rm GL_{\infty}} (\rmT)) =  \Sigma_{\T}^{\infty} (\colimn \BN_{\rm GL_n} (T_n)) \ra \tilde{\mathbf K},
\end{equation} \ee
     where $\tilde {\mathbf K}$ is the motivic $\Omega_{\T}$-spectrum whose $0$-th space is $\BGL_{\infty}$. Let $\phi$ denote the above map
     in ~\eqref{compos.map}.
     
Let $h$ denote the motivic cohomology theory defined by the mapping cone of the above map $\phi$. Then, for every motivic space $\rmX$, we obtain a long-exact sequence:
\be \begin{equation}
\label{long.exact.1}
 \cdots \ra [\rmX, \rmQ (\BN_{\rm GL_{\infty}}  (\rmT ) ) ] {\overset {\bar q_* }\ra } [\rmX, \BGL _{\infty}] {\overset c \ra}  h^{0,0}(\rmX) \ra \cdots
\end{equation} \ee
For each $n \ge 0$, let
\be \begin{equation}
\label{u}
 u_n \eps [\BGL_{n}, \BGL_{\infty}]
 \end{equation} \ee
\vskip .1cm \noindent
be the class of the map induced by the imbedding $\GL_n \ra \GL_{\infty}$. Then it suffices to show that each such $u_n$ is in the image of the induced map $\bar q_*$, which is equivalent to showing that $c(u_n) =0$. (To see this let $v_n: \BGL_{n} \ra  \rmQ (\BN_{\rm GL_{\infty}} (\rmT ) )$ be such that $\bar q_*([v_n]) = [\bar q \circ v_n] = {\it u}_n$.
Here, if $\alpha$ is a map,  $[\alpha]$ denotes the stable homotopy class of  $\alpha$. Since $\rmX$ is assumed to be compact,
$[\rmX, BGL_{\infty}]= \colimn [\rmX, BGL_n]$. Therefore, giving an $\alpha \eps [\rmX, \BGL_{\infty}]$ is equivalent to giving 
 an $\alpha_n: \rmX \ra \BGL_{n}$ (for some $n$), so that $\alpha = u_n \circ \alpha _n$. Now let $\beta_n = v_n \circ \alpha_n$. 
 Then $\bar q_*([\beta_n ]) = [\bar q\circ v_n \circ \alpha_n] = [{\it u}_n \circ\alpha_n] = [\alpha]$.)
\vskip .2cm
Now we observe the commutative diagram, where ${\rm p}_n: {\rm BN_{\rm GL_n} (T_n)} \ra {\rm BGL_n}$ is the obvious map induced by the
imbedding ${\rm N_{\rm GL_n} (T_n)} \ra {\rm GL_n}$:
 \be \begin{equation}
 \label{comm.diagm}
\xymatrix{{[\BGL _{n}, \rmQ (\BN_{\rm GL_{\infty}}  (\rmT ))]} \ar@<1ex>[r]^(.6){\bar q_*}  \ar@<1ex>[d]^{\rmp_n^*} & {\tilde \rmK^{0,0}(\BGL _{n})} \ar@<1ex>[r]^c \ar@<1ex>[d]^{\rmp_n^*} & {{\rm h}^{0, 0}(\BGL _{n})} \ar@<1ex>[d]^{\rmp_n^*}\\
          {[\BN _{\rm GL_n} (\rmT_n ), \rmQ (\BN_{\rm GL_{\infty}}  (\rmT ) )]} \ar@<1ex>[r]^(.6){\bar q_*} & {\tilde \rmK^{0,0}(\BN_{\rm GL_n}  (\rmT_n ) )} \ar@<1ex>[r]^c & {{\rm h}^{0, 0}(\BN _{\rm GL_n} (\rmT_n ) ) }}
 \end{equation} \ee
 \vskip .2cm
Recall that  $\rmp_n^*$ is a monomorphism, by Proposition ~\ref{split.0}. Therefore, now it suffices to prove $\rmp_n^*(c({\it u}_n))=0$, for each $n$. But the commutativity of the above diagram, shows that this is equivalent to showing that $c(\rmp_n^*({\it u}_n)) =0$. 
At this point, we observe the commutative diagram, where ${\it i_n}$ and $j$ are the obvious maps:
\[ \xymatrix{{\rm BN_{\rm GL_n} (T_n)} \ar@<1ex>[r]^{\rm p_n} \ar@<1ex>[d]^{ i_n} &{\rm BGL_n} \ar@<1ex>[d]^{u_n}\\
             {\rm BN_{\rm GL_{\infty}} (T)} \ar@<1ex>[r]^{\rm p} \ar@<1ex>[d]^{j} & {\rm BGL_{\infty}} \ar@<1ex>[d] \ar@<1ex>[dr]^{id}\\
             {\rm Q(BN_{\rm GL_{\infty}} (T))} \ar@<1ex>[r]^{\rm Q(p)} & {\rm Q(BGL_{\infty})} \ar@<1ex>[r]^q & {\rm BGL_{\infty}}.}
\]
The map $\BGL_{\infty} \ra \rmQ(\BGL_{\infty})$ is provided by taking the colimit of the maps $\BGL_{\infty} \ra \Omega_{\T}^n\Sigma_{\T}^n(\BGL_{\infty})$,
 while the map $\rmQ(\BGL_{\infty}) \ra \BGL_{\infty}$ is provided by the fact $\BGL_{\infty}$ is an infinite $\T$-loop space.
 The fact that the composite map $\BGL_{\infty} \ra \rmQ(\BGL_{\infty}) \ra \BGL_{\infty}$ is homotopic to the identity map follows
 from the adjunction between $\wedge \T$ and $\Omega_{\T}$, along with the $\Omega_{\T}$-infinite loop space structure on $\BGL_{\infty}$.
In view of the commutative diagram above,  $\rmp_n^*({\it u}_n) = {\it u_n}  \circ {\rm p_n} = q \circ {\rm Q(p)} \circ {\it j} \circ {\it i_n} = {\bar q}_*({\it j} \circ {\it i_n})$.
Since the two  rows in the diagram ~\eqref{comm.diagm} are exact, it follows that indeed $c(\bar q_*({\it j} \circ i_n)) =0$, thereby 
completing the proof that $c({\rm p}_n^*({\it u_n}))=0$.
\end{proof}
\vskip .2cm
 Let $\lambda: \rmQ (B{\mathbb G}_m) \ra \rmQ (\BGL _{\infty}) \ra \BGL _{\infty}$ denote the map considered in ~\eqref{splitting.map}. 
\begin{proposition}
	\label{key.comm.triangle.0} 
	\begin{enumerate}[\rm(i)]
	\item Assume the base field ${\it k}$ is of characteristic $0$. 
	Let $\rmX$ denote any pointed motivic space.
	Then, there exists a map $\zeta: \rmQ (\BN_{\rm GL_{\infty}}  (\rmT ) ) \ra \rmQ (\rmB{\mathbb G}_m) = \rmQ (\P^{\infty} )$, so that the triangle
\be	\begin{equation}
	\xymatrix{ {[\rmX, \rmQ ( \BN_{\rm GL_{\infty}} (\rmT ) )]} \ar@<1ex>[r]^{\zeta_*} \ar@<1ex>[dr]_{\bar q_*} & {[ \rmX, \rmQ (B{\mathbb G}_m)]} \ar@<1ex>[d]^{\lambda_*} \\
		& {[\rmX, \BGL _{\infty} ]}}
	\end{equation}	\ee
	commutes.
	\item Assume the base field ${\it k}$ is perfect and of positive characteristic $p$. Then the 
    same conclusion holds after inverting the prime $p$.
    \end{enumerate}
\end{proposition}
\vskip .2cm \noindent
{\bf Proof of Theorem ~\ref{mainthm.1}}. 
Before we prove the above proposition, we proceed to show how to complete the proof of Theorem ~\ref{mainthm.1}, given the above Proposition.
We simply observe that, for a compact object $\rmX$, the composition of the maps in
\[  [\rmX, \rmQ (\BN_{\rm GL_{\infty}} (\rmT ) )] {\overset {\zeta_*} \ra} [ \rmX, \rmQ (\rmB {\mathbb G}_m))] {\overset {\lambda_*} \ra} [ \rmX, \BGL _{\infty} ] \]
is a surjection, thereby proving that the map $\lambda_*$ is also surjection, which completes
the proof of the theorem.
\vskip .2cm 
\subsection{}
Therefore, it suffices to prove Proposition ~\ref{key.comm.triangle.0}, which we proceed to do presently. 
Moreover, we will only consider the characteristic $0$ case explicitly, as the positive characteristic case follows
exactly along the same lines, once the characteristic is inverted.
Let $\St_{2n,n}$ denote the
{\it Stiefel variety} of $n$-frames (that is, $n$-linearly independent vectors) $({\mathbf v}_1, \cdots, {\mathbf v}_n)$ in ${\mathbb A}^{2n}$. The group $\GL_n$ acts on this variety, by acting on such frames by sending $({\mathbf v}_1, \cdots, {\mathbf v}_n)$ to $({\mathbf v}_1, \cdots, {\mathbf v}_n)*g$, $g \eps \GL_n$. 
(We view this as a right action because the Stiefel variety $\St_{2n,n}$ identifies with the 
variety of all $2n\times n$-matrices of rank $n$, with each vector ${\mathbf v}_i$ in an $n$-frame $({\mathbf v}_1, \cdots, {\mathbf v}_n)$ written
as the $i$-th column.)
This is a free-action and the quotient is
the {\it Grassmannian}, ${\rm Grass}_{2n,n}$. The Stiefel variety $\St_{2n,n}$ is an open sub-variety of the affine space ${\mathbb A}^{2n^2}$ and the complement has codimension $n+1$ in ${\mathbb A}^{2n^2}$.
\vskip .1cm
Next we consider the ind-scheme:
\be \begin{equation}
 \label{A.infty}
 {\mathbb A}^2 {\overset {} \ra} {\mathbb A}^4 {\overset {} \ra}{\mathbb A}^6 \cdots {\overset {} \ra} {\mathbb A}^{2n} {\overset {i_n} \ra} {\mathbb A}^{2n+2} \cdots,
\end{equation} \ee
where the closed immersion ${\mathbb A}^{2n} \ra {\mathbb A}^{2n+2}$  sends $(x_1, \cdots, x_{2n})$ to $(x_1, \cdots, x_{2n}, 0, 0)$.  Let ${\mathbf e}_i$, $i=1, \cdots, 2n, 2n+1, 2n+2$
denote the standard basis vectors in ${\mathbb A}^{2n+2}$. Then we obtain a closed immersion
\be \begin{equation}
 \label{closed.imm.St}
  i_n:\St_{2n,n} \ra \St_{2n+2, n+1},
\end{equation} \ee
by sending an $n$-frame $({\mathbf v}_1, \cdots, {\mathbf v}_n)$ in ${\mathbb A}^{2n}$ to the $n+1$-frame $({\mathbf v}_1, \cdots, {\mathbf v}_n, \bar {\mathbf e}_{n+1}= {\mathbf e}_{2n+1}+{\mathbf e}_{2n+2})$.
 This induces
a closed immersion of the Grassmannians, ${\rm Grass}_{2n,n} \ra {\rm Grass}_{2n+2, n+1}$.
Therefore, the discussion in section ~\ref{geom.class.sp} shows that the above Stiefel varieties may be used as finite dimensional approximations to 
$\EGL_n$ and the Grassmannian, ${\rm Grass}_{2n,n}$, could be used as a finite dimensional approximation to 
$\BGL_n$.
\vskip .2cm
Next we make the following identifications.

\be \begin{align}
  \label{key.idents}
      \BGL_n^{gm,n} &= \St _{2n,n}/\GL _{n},\\
      \BN_{\GL_n}(T_n) ^{gm,n} &= \St _{2n,n}/{\rm N_{\GL_n}(T_n)}, \mbox{ and } \notag\\
      {\widetilde {\BN_{\GL_n}(T_n) ^{gm, n}}} &= \St _{2n,n}/({\mathbb G}_m \times  {\rm N_{\GL_{n-1}}(T_{n-1})}), \notag
\end{align} \ee
where we imbed $\GL_{n-1}$ in $\GL_n$ as the last  $(n-1)\times (n-1)$-block, then
imbed $\rmN_{\GL_n}(\rmT_n)$ in $\GL_n$, and imbed ${\mathbb G}_m \times \rmN_{\GL_{n-1} (\rmT_{n-1})}$ in ${\mathbb G}_m \times \GL_{n-1} $. Now $ {\mathbb G}_m \times \rmN_{\GL_{n-1}} (\rmT_{n-1}) $ is a subgroup of index $n$ in $\rmN_{\GL_n} (\rmT_n)$, so that the
projection $r_n:{\widetilde {\BN_{\GL_n} (\rmT_n)^{gm,n}}}  \ra \BN_{\GL_n} (\rmT_n)^{gm,n}$ is a finite \'etale cover of degree $n$. The fact that the terms appearing on the right-hand-sides  in ~\eqref{key.idents} are indeed approximations to
the classifying spaces of the corresponding linear algebraic groups, follows from the discussion in section ~\ref{geom.class.sp}. 
\vskip .1cm
Next consider the map $\St_{2n,n} \ra \St_{2n, 1}$ sending an $n$-frame $({\mathbf v_1}, {\mathbf v}_{2}, \cdots, {\mathbf v}_n)$ to ${\mathbf v}_1$.
Clearly this factors through the quotient of $\St_{2n,n}/1 \times \GL_{n-1}$, where $\GL_{n-1}$ acts only on the last $n-1$-vectors
in the $n$-frame $({\mathbf v}_1, \cdots, {\mathbf v}_{n-1}, {\mathbf v}_n)$. Therefore, we obtain the map 
\be \begin{equation}
  \label{phi.n}
  \phi_n: \St_{2n,n}/(1 \times \GL_{n-1}) \ra \St_{2n,1}.
    \end{equation} \ee
Moreover, the above map $\phi_n$ is compatible with the obvious action of ${\mathbb G}_m$ on $\St_{2n,n}/\GL_{n-1}$ where it acts on
the vector ${\mathbf v}_1$ in an $n$-frame, $({\mathbf v}_1, \cdots, {\mathbf v}_{n-1}, {\mathbf v}_n)$, and it acts
on the $1$-frame ${\mathbf v}$ in $\St_{2n,1}$. Taking quotients, this defines the map
\be \begin{equation}
  \label{bphi.n}
  \bar \phi_n: \St_{2n,n}/({\mathbb G}_m \times \GL_{n-1}) \ra \St_{2n,1}/{\mathbb G}_m.
    \end{equation} \ee
One may then observe the commutative square:
\be \begin{equation}
 \label{compatible.phi}
\xymatrix{{\St_{2n,n}/({\mathbb G}_m \times \GL_{n-1})} \ar@<1ex>[r]^(.6){\bar \phi_n} \ar@<1ex>[d] & {\St_{2n,1}/{\mathbb G}_m} \ar@<1ex>[d]\\
           {\St_{2n+2,n+1}/({\mathbb G}_m \times \GL_{n})} \ar@<1ex>[r]^(.6){\bar \phi_{n+1}}  & {\St_{2n+2,1}/{\mathbb G}_m},}
\end{equation} \ee
where the left vertical map is the closed immersion defined by $i_n:\St_{2n,n} \ra \St_{2n+2, n+1}$ and the 
right vertical map is induced by the closed immersion $\St_{2n,1} \ra \St_{2n+2, 1}$.

 One may also observe that clearly $\St_{2n,1}/{\mathbb G}_m$ is an approximation to the classifying space of ${\mathbb G}_m$,
 so that we will let
\be \begin{equation}
   \label{BGm.approx.n}
   \B{\mathbb G}_m^{gm, n} = \St_{2n,1}/{\mathbb G}_m.
\end{equation} \ee
We also let 
\be \begin{align}
\label{bar.p.n}
  u_n: {\widetilde {\BN_{\GL_n}(\rmT_{n})^{gm,n}}}= \St_{2n,n}/({\mathbb G}_m \times N_{\GL_{n-1}}(\rmT_{n-1})) &\ra \St_{2n,n}/({\mathbb G}_m \times \GL_{n-1}), \mbox{ and}\\
\bar u_n =\bar \phi_n \circ u_n: {\widetilde {\BN_{\GL_n}(\rmT_{n})^{gm,n}}}= \St_{2n,n}/({\mathbb G}_m \times N_{\GL_{n-1}}(\rmT_{n-1})) &\ra \St_{2n,1}/{\mathbb G}_m = \B{\mathbb G}_m^{gm,n}. \notag 
 \end{align} \ee
\vskip .1cm
Then we also obtain the commutative diagram:
\be \begin{equation}
 \label{compatible.p}
\xymatrix{{\St_{2n,n}/({\mathbb G}_m \times N_{\GL_{n-1}}(\rmT_{n-1}))} \ar@<1ex>[r]^(.55){u_n}\ar@<1ex>[d] & {\St_{2n,n}/({\mathbb G}_m \times \GL_{n-1})} \ar@<1ex>[r]^(.6){\bar \phi_n} \ar@<1ex>[d] & {\St_{2n,1}/{\mathbb G}_m} \ar@<1ex>[d]\\
           {\St_{2n+2,n+1}/({\mathbb G}_m \times N_{\GL_{n}}(\rmT_n))} \ar@<1ex>[r]^(.55){u_{n+1}}& {\St_{2n+2,n+1}/({\mathbb G}_m \times \GL_{n})} \ar@<1ex>[r]^(.6){\bar \phi_{n+1}}  & {\St_{2n+2,1}/{\mathbb G}_m}.}
\end{equation} \ee

\vskip .2cm
As pointed out in the introduction,  apart from the transfer for passage from 
$\BGL_n^{gm,n}$ to $\BN_{\GL_n}(\rmT_n)^{gm,n}$, for the proof of Theorem ~\ref{mainthm.1},  one also 
needs to invoke a transfer map for the finite \'etale map 
\[r_n : {\widetilde {\BN_{\GL_n}(\rmT_{n})^{gm,n}}}= \St_{2n,n}/({\mathbb G}_m \times N_{\GL_{n-1}}(\rmT_{n-1})) \ra  \St_{2n,n}/ N_{\GL_{n}}(\rmT_{n}) = {\BN_{\GL_n}(\rmT_{n})^{gm,n}}.\] 
We also need to know that such a transfer map has reasonable properties, like compatibility with base-change,
and agreement with Gysin maps defined on orientable generalized motivic cohomology theories. The purpose of the last short section of the paper 
is to set-up such a transfer and establish these basic properties for it: see ~\eqref{rel.transf} for the definition of such a transfer.
Let 
\be \begin{equation}
\label{transf.tau.0}
\tau_n: \Sigma_{\T}^{\infty} \BN_{\GL_n} (\rmT_n)^{gm,n}_+   \ra \Sigma_{\T}^{\infty}({\widetilde {\BN_{\GL_n}(\rmT_n)^{gm,n}}})_+
\end{equation} \ee
denote the corresponding transfer defined as in ~\eqref{rel.transf},   and let 
\be \begin{equation}
\label{zeta.n}
\zeta_n: \Sigma_{\T}^{\infty} \BN_{\GL_n}(\rmT_n)^{gm,n}_+ {\overset {\tau_n} \ra } \Sigma_{\T}^{\infty}({\widetilde {\BN_{\GL_n}(\rmT_n)^{gm,n}}})_+ {\overset {\pi_n} \ra } \Sigma_{\T}^{\infty}\rmB {\mathbb G}_m^{gm,n} {\overset {j_n} \ra} \Sigma_{\T}^{\infty}\rmB {\mathbb G}_m
\end{equation} \ee
denote the composition, where the map $\pi_n $ is the composition of the map $\Sigma_{\T}^{\infty}\bar u_{n+}$ followed by the
map that sends the base point $+$ to the base point of $\rmB {\mathbb G}_m^{gm,n}$ as in section ~\ref{change.base.pts}.
The last map, denoted $j_n$, is the obvious one sending a finite dimensional approximation of $\rmB {\mathbb G}_m$ to the direct limit of such approximations.
Let 
\be \begin{equation}
\label{bar.q.n}
\bar q_n: \rmQ (\BN_{\GL_n}(\rmT_n)^{gm,n}_+) \ra \BGL_{\infty}
\end{equation} \ee
denote the composition 
\[\rmQ (\BN_{\GL_n} (\rmT_n)^{gm,n}_+)  \ra \rmQ( \BN_{\GL_{\infty}}(\rmT)) {\overset {\rmQ(p)} \ra}\rmQ(\BGL _{\infty}) {\overset q \ra} \BGL _{\infty}.\]
Then a key result is the following Proposition, which we show also proves Proposition ~\ref{key.comm.triangle.0}.
\begin{proposition}
\label{key.diagms}
 Assume the above situation. Then the following diagrams commute:
 \be \begin{equation}
 \label{key.comm.triangle.1}
 \xymatrix{ {\Sigma_{\T}^{\infty} \BN_{\GL_n} (\rmT_n)^{gm,n}_+} \ar@<1ex>[dr]^{\zeta_n} \ar@<1ex>[d]   \\
 	{\Sigma_{\T}^{\infty} \BN_{\GL_{n+1}}(\rmT_{n+1})^{gm,n+1}_+} \ar@<1ex>[r]_(.65){\zeta_{n+1}} & {\Sigma_{\T}^{\infty}\rmB{\mathbb G}_m}}
 \end{equation}\ee 
\be \begin{equation}
\label{key.comm.triangle.2}
 \xymatrix{{[\rmX, \rmQ (\BN_{\GL_n}(\rmT_n)^{gm,n}_+)]} \ar@<1ex>[r]^(.6){\zeta_{n*}}  & {[\rmX, \rmQ (\rmB{\mathbb G}_m)]}  \ar@/^4pc/[dd]^{\lambda_*}\\
 {[\Sigma _{\T}^{\infty}\rmX, \Sigma_{\T}^{\infty}(\BN_{\GL_n}(\rmT_n)^{gm,n}_+)]} \ar@<1ex>[u]^{\simeq} \ar@<1ex>[r]^(.6){\zeta_{n*}} \ar@<1ex>[dr]_(.6){\bar q_{n*}} & {[\Sigma_{\T}^{\infty}\rmX, \Sigma_{\T}^{\infty}\B{\mathbb G}_m]}  \ar@<-1ex>[u]^{\simeq} \ar@<1ex>[d]^{\lambda'_*} \\
 	& {{\tilde \rmK^{0,0}(\rmX)} =[\rmX, \BGL_{\infty}]} }
\end{equation} \ee 
\end{proposition} \ee \noindent
where $\lambda':\Sigma_{\T}^{\infty}\B{\mathbb G}_m \ra f_1({\mathbf K}) =\tilde {\mathbf K}$ is the map of spectra corresponding to the infinite loop-space map $\lambda: \rmQ(\B{\mathbb G}_m) \ra \BGL_{\infty}$. 
The left vertical map in ~\eqref{key.comm.triangle.1} is the obvious map induced by the closed immersion $i_n$ in ~\eqref{closed.imm.St}. Moreover,
$[\quad, \quad ]$ in the middle row of ~\eqref{key.comm.triangle.2} denotes Hom in the motivic  stable homotopy category, while $[\quad, \quad ]$ in the top row and the bottom row denotes
Hom in the unstable pointed motivic  homotopy category. 
\begin{proof}
We will first prove the commutativity of the triangle in ~\eqref{key.comm.triangle.1}. For this, one begins with the cartesian square (which also defines ${\rm P}_n$ and the map $r_n'$):
\be \begin{equation}
\label{Pn}
  \xymatrix{{\rm {P_n}}  \ar@<1ex>[r]^(.2){\tilde i} \ar@<1ex>[d]^{r_{n}'} & {\widetilde {\BN_{\GL_{n+1}}(\rmT_{n+1})^{gm,n+1}} ={\St_{2n+2, n+1}/({\mathbb G}_m} \times {\rm N_{\GL_n}(\rmT_n)})} \ar@<1ex>[d]^{r_{n+1}}\\
  	        {\BN_{\GL_n}(\rmT_n)^{gm,n}=St_{2n,n}/N_{\GL_n}(\rmT_n) } \ar@<1ex>[r]^(.4)i & {\BN_{\GL_{n+1}}(\rmT_{n+1})^{gm, n+1} ={\St_{2n+2, n+1}/{\rm N_{\GL_{n+1}}(\rmT_{n+1})}} }.}
\end{equation} \ee
Observe that the right vertical map, and therefore, the left vertical map also is a finite \'etale map of degree $n+1$.
By Proposition ~\ref{base.change} (which shows the naturality of the transfer with respect to base-change), we observe that 
the square below homotopy commutes:
\be \begin{equation}
\label{base.change.square}
\xymatrix{{\Sigma_{\T}^{\infty}{\rmP_{n,+}}} \ar@<1ex>[r]^{\Sigma_{\T}^{\infty}\tilde i_+} &{ \Sigma_{\T}^{\infty}	\widetilde {\BN_{\GL_{n+1}}	(\rmT_{n+1})}_+ ^{\rm gm,n+1}} \\
	{\Sigma_{\T}^{\infty}\BN_{\GL_n}(\rmT_n)_+^{gm,n}}\ar@<1ex>[r]^{\Sigma_{\T}^{\infty}i_+} \ar@<1ex>[u]^{\tau _n'} & {\Sigma_{\T}^{\infty}{\BN_{\GL_{n+1}}(\rmT_{n+1})_+}^{\rm gm,n+1}} \ar@<1ex>[u]^{\tau_{n+1}} .}
\end{equation} \ee
\vskip .1cm \noindent 

Then a straightforward calculation, as discussed below, shows that ${\rm {P_n}} = {\widetilde {\BN_{\GL_n}(\rmT_{n})^{gm,n}}} \sqcup \BN_{\GL_n}(\rmT_n)^{gm,n}$: the main observation here is that under the identifications in ~\eqref{key.idents}, the map $\BN_{\GL_n}(T_n)^{gm,n} \ra \BN_{\GL_{n+1}}(T_{n+1})^{gm, n+1}$ lifts to ${\widetilde {\BN_{\GL_{n+1}}(T_{n+1})}} ^{gm, n+1}$, which provides
the required splitting. In fact, this splitting may be described in more detail as follows. Observe first that the imbedding $i:\St_{2n,n} \ra \St_{2n+2, n+1}$ is
defined by sending an $n$-frame, $({\mathbf v}_1, \cdots, {\mathbf v}_n)$ in $\St_{2n, n}$ to
the $n+1$-frame, $({\mathbf v}_1, \cdots, {\mathbf v}_n, \bar {\mathbf e}_{n+1})$ in $\St_{2n+2, n+1}$, where $\bar {\mathbf e}_{n+1}$ is the  nonzero
vector chosen as in ~\eqref{closed.imm.St}  that lies in the ambient affine space ${\mathbb A}^{2n+2}$ and is outside of ${\mathbb A}^n \subseteq {\mathbb A}^{n+2}$.
The induced map $\St_{2n,n}/{\rm N_{\GL_n}(\rmT_n)} \ra \St_{2n+2, n+1}/{\rm N_{\GL_{n+1}}}$ is the map
$i: \BN_{\GL_n}(\rmT_n)^{gm,n} \ra \BN_{\GL_{n+1}}(\rmT_{n+1})^{\rm gm,n+1}$ appearing in ~\eqref{Pn}.
\vskip .1cm
One may see that from $i$ one obtains an induced map
\be \begin{equation}
 \label{splitting.1}
 \St_{2n,n}/({\mathbb G}_m \times N_{\GL_{n-1}}(\rmT_{n-1})) \ra  \St_{2n+2,n+1}/({\mathbb G}_m \times N_{\GL_{n}}(\rmT_{n})),
\end{equation} \ee
since the action of ${\mathbb G}_m \times {\rm N}_{\GL_{n-1}}(\rmT_{n-1}) $ on  $\St_{2n,n}$ and the action of 
${\mathbb G}_m \times {\rm N}_{\GL_{n}}(\rmT_{n})$ on $\St_{2n+2, n+1}$ are compatible. (For this
one identifies $\St_{2n,n}$ as imbedded in $\St_{2n+2, n+1}$ using the imbedding $i_n$ considered in ~\eqref{closed.imm.St}, and 
identifies the group ${\mathbb G}_m \times {\rm N}_{\GL_{n-1}}(\rmT_{n-1}) $ with the subgroup 
${\mathbb G}_m \times {\rm N}_{\GL_{n-1}}(\rmT_{n-1}) \times 1 $ of ${\mathbb G}_m \times {\rm N}_{\GL_{n}}(\rmT_{n})$.)
Moreover, one may see that this map is a closed immersion and that one obtains a commutative triangle:
\be \begin{equation}
 \label{splitting.2}
 \xymatrix{{\St_{2n,n}/({\mathbb G}_m \times N_{\GL_{n-1}}(\rmT_{n-1}))} \ar@<1ex>[rr] \ar@<1ex>[dr]&& {{\rm P}_n} \ar@<1ex>[dl]^(.3){r_n'}\\
                       & {\St_{2n,n}/N_{\GL_{n}}(\rmT_{n})}.}
\end{equation} \ee
The left inclined map is a finite \'etale map of degree $n$, while the right inclined map is a finite \'etale map of degree $n+1$. Since the top horizontal
map is also  \'etale (see \cite[Chapter I, Corollary 3.6]{Mi}) and a closed immersion, it is the open (and closed) imbedding of a connected component in ${\rm P}_n$. Let the complement in 
${\rm P}_n$ of ${\St_{2n,n}/({\mathbb G}_m \times N_{\GL_{n-1}}(\rmT_{n-1}))}$ be denoted ${\rm C}_n$. Then the induced map ${\rm C}_n \ra {\St_{2n,n}/N_{\GL_{n}}(\rmT_{n})}$
is a bijective and finite \'etale map, so is an isomorphism, showing that ${\St_{2n,n}/N_{\GL_{n}}(\rmT_{n})} =
\BN_{\GL_n}(\rmT_n)^{gm,n}$ is a split summand of $\St_{2n+2,n+1}/({\mathbb G}_m \times {\rm N_{\GL_n}(\rmT_n)}) = {\widetilde {\BN_{\GL_{n+1}}(\rmT_{n+1})^{gm,n+1}}}$.
(One may also obtain an explicit description of the above splitting as given by  sending the $n$-frame  
$({\mathbf v}_1, \cdots, {\mathbf v}_n)$ in $\St_{2n, n}$ to
the $n+1$-frame, $(\bar {\mathbf e}_{n+1}, {\mathbf v}_1, \cdots, {\mathbf v}_n)$ in $\St_{2n+2, n+1}$.
${\mathbb G}_m$ acts on $\bar {\mathbf e}_{n+1}$ by multiplication by scalars. Let $s$ denote this imbedding of ${\St_{2n,n}/N_{\GL_{n}}(\rmT_{n})}$ in
$\St_{2n+2,n+1}/({\mathbb G}_m \times {\rm N_{\GL_n}(\rmT_n)})$.)
\vskip .2cm
Moreover, one observes from the above description of the splitting ${\rm {P_n}} = {\widetilde {\BN_{\GL_n}(\rmT_{n})^{gm,n}}} \sqcup \BN_{\GL_n}(\rmT_n)^{gm,n}$, that
under the composite map 
\be \begin{multline}
  \label{trivial.map}
\begin{split}
\bar u_{n+1}\circ \tilde i: \rmP_n {\overset {\tilde i} \ra } {\widetilde {\BN_{\GL_{n+1}}(\rmT_{n+1})}}^{gm,n+1} = {\St_{2n+2, n+1}/({\mathbb G}_m \times {\rm N_{\GL_n}(\rmT_n)}) } {\overset {{\it u}_{n+1} } \ra}  {\St_{2n+2, n+1}/({\mathbb G}_m \times {\GL_n}) }\\
 {\overset {\bar \phi_n} \ra}  {\St_{2n+2, 1}/{\mathbb G}_m}
\end{split}
\end{multline}
the copy of $\BN_{\GL_n}(\rmT_n)^{gm,n} = {\St_{2n, n}/{\rm N_{GL_n}(\rmT_n)}}$ in ${\rm P}_n$ (under the above splitting of ${\rm P}_n$) is sent to the base point. 
(Observe that the $n$-frames $({\mathbf v}_1, \dots, {\mathbf v}_n)$ coming from $\St_{2n, n}$, under the above imbedding $s$, get sent to the last $n$-frames.) Since the diagram
\[ \xymatrix{ {\rmB {\mathbb G}_m^{gm, n}} \ar@<1ex>[r] &{\rmB {\mathbb G}_m^{gm, n+1}}\\
{\widetilde {\BN_{\GL_n} (\rmT_{n})^{gm,n}}}	\ar@<1ex>[u]^{\bar u_n} \ar@<1ex>[r]  & {\widetilde {\BN_{\GL_{n+1}} (\rmT_{n+1} ) }^{gm,n+1} }   \ar@<1ex>[u]^{\bar u_{n+1}} }\]
\vskip .1cm \noindent
(which is the same as the diagram ~\eqref{compatible.p}) also commutes, combining these diagrams and composing with the inclusions into $\rmB {\mathbb G}_m$ proves 
the commutativity of the triangle ~\eqref{key.comm.triangle.1}. In fact, the commutativity of the triangle ~\eqref{key.comm.triangle.1} 
is equivalent to the statement that the
composition of maps along the left column followed by the top inclined map is the same (up to homotopy) as the
composition of the bottom map followed by the maps in the right column in the 
 big diagram:
\[\xymatrix{&& {\Sigma_{\T}^{\infty}\B{\mathbb G}_{m}}\\
            {\Sigma_{\T}^{\infty}\B{\mathbb G}_{m}^{\rm gm,n}} \ar@<1ex>[rr]\ar@<1ex>[urr]^{j_n} &&{\Sigma_{\T}^{\infty}\B{\mathbb G}_{m}^{\rm gm, n+1}} \ar@<1ex>[u]^{j_{n+1}}\\
            {\Sigma_{\T}^{\infty}\widetilde {\BN_{\GL_n}	(\rmT_{n})}_+ ^{\rm gm,n}} \ar@<1ex>[r] \ar@<1ex>[u]^{ \pi_n}& {\Sigma_{\T}^{\infty}\rmP_{n,+}} \ar@<1ex>[r]^(.3){\Sigma_{\T}^{\infty}\tilde i_+} &{\Sigma_{\T}^{\infty}\widetilde {\BN_{\GL_{n+1}}(\rmT_{n+1})}_+ ^{\rm gm,n+1}} \ar@<1ex>[u]^{ \pi_{n+1}}\\
            {\Sigma_{\T}^{\infty} {\BN_{\GL_n}(\rmT_{n})}_+ ^{\rm gm,n}} \ar@<1ex>[rr]\ar@<1ex>[u]^{\tau_n} \ar@<1ex>[ur]^{\tau'_n}& &{\Sigma_{\T}^{\infty}{\BN_{\GL_{n+1}}(\rmT_{n+1})_+ ^{\rm gm,n+1}}} \ar@<1ex>[u]^{\tau_{n+1}} .}
\]
The commutativity of the part of the above diagram on and above the second row has already been proven. The bottom square on the
right commutes by ~\eqref{base.change.square}. Proposition ~\ref{base.change}(i)  proves that $\tau_n'= \tau_n+id$, where $id $ denotes the identity map of ${\Sigma_{\T}^{\infty} {\BN_{\GL_n}(\rmT_{n})}_+ ^{\rm gm,n}}$.
Finally the observation in ~\eqref{trivial.map} completes the proof.
\vskip .2cm
Next we consider the commutativity of the diagram ~\eqref{key.comm.triangle.2}. Since the top square there evidently commutes, it suffices 
to consider the commutativity of the bottom triangle there. {\it The key observation is that, in order to prove the
commutativity of the bottom triangle in ~\eqref{key.comm.triangle.2}, it suffices to take $\rmX= \BN_{\GL_n}(\rmT_n)^{gm,n}_+$ and 
show that the triangle commutes for 
the class $u \eps [\rmX, \rmQ (\BN_{\GL_n}(\rmT_n)^{gm,n}_+)]$ denoting the class corresponding to the identity map
$\Sigma_{\T}^{\infty}\rmX_+ \ra \Sigma_{\T}^{\infty}\rmX_+$. }
\vskip .1cm
Let $u$ denote the class considered in the last line. Then $\bar q_{n*}(u) = (q\circ \rmQ(\rmp))_*(u)$ denotes the class of the vector bundle of
 rank $n$ associated to the principal $\rmN_{\GL_n}(\rmT_n)$-bundle over $\BN_{\GL_n}(\rmT_n)^{gm,n}$. One may see this readily as follows:
 First the natural map $p: \BN_{\GL_n}(\rmT_n)^{gm,n} \ra \BGL_{\infty}$ corresponds to the rank $n$ vector bundle over
 $\BN_{\GL_n}(\rmT_n)^{gm,n}$ associated  to the principal $\rmN_{\GL_n}(\rmT_n)$-bundle over $\BN_{\GL_n}(\rmT_n)^{gm,n}$. Then the
 homotopy commutative diagram:
 \be \begin{equation}
      \xymatrix{{\BN_{\GL_n}(\rmT_n)^{gm,n}} \ar@<1ex>[r]^{p} \ar@<1ex>[d]^u & {\BGL_{\infty}} \ar@<1ex>[d] \ar@<1ex>[dr]^{id}\\
                {\rmQ(\BN_{\GL_n}(\rmT_n)^{gm,n})} \ar@<1ex>[r]^{\rmQ({\it p})}  & {\rmQ(\BGL_{\infty})} \ar@<1ex>[r]^q & {\BGL_{\infty}}} \notag
\end{equation} \ee 
completes the proof. {\it We will denote this vector bundle over}
\be \begin{equation}
\label{alpha}
\BN_{\GL_n}(\rmT_n)^{gm,n} =\St_{2n,n}/N_{\GL_n}(\rmT_n) \mbox{ by } \alpha.
\end{equation}
\vskip .2cm
{\it Let $\beta$ denote the line bundle
 associated to the principal ${\mathbb G}_m$-bundle} 
 \be \begin{equation}
 \label{beta}
 \St _{2n,n}/(1 \times  \rmN_{\GL_{n-1}}(\rmT_{n-1})) \ra {\St}_{2n,n}/({\mathbb G}_m  \times \rmN_{\GL_{n-1}}(\rmT_{n-1}))  = {\widetilde {\BN_{\GL_n}(\rmT_n)^{gm, n}}}.
 \end{equation} \ee
 Then $\lambda'_*(\zeta_{n*}(u))$ is the image of $\beta$ under the transfer map $\tau_n^*:\tilde \rmK^{0,0}( {\widetilde {\BN_{\GL_n}(\rmT_n)^{gm, n}}}) \ra \tilde \rmK^{0,0}(\BN_{\GL_n}(\rmT_n)^{gm,n})$,
 where $\zeta_n$ is the map in ~\eqref{zeta.n}. 
 This results from the following observations:
 \begin{enumerate}[\rm(i)]
 \item the map $\lambda'\circ j_n: \Sigma_{\T}^{\infty}\B{\mathbb G}_{}^{\rm gm, n} \ra \Sigma_{\T}^{\infty}\B{\mathbb G}_{m}^{\rm gm} \ra f_1({\mathbf K}) = \tilde {\mathbf K}$ corresponds to a map
 $\tilde \lambda'_n:\St_{2n,1}/{\mathbb G}_m= \B{\mathbb G}_{m}^{\rm gm, n} \ra \BGL_{\infty}$
 \item 
 the map $\tilde \lambda'_n: \St_{2n,1}/{\mathbb G}_m= \B{\mathbb G}_{m}^{\rm gm, n} \ra \BGL_{\infty}$ corresponds
 to the line bundle on 
 \newline \noindent
 $\St_{2n, 1}/{\mathbb G}_m$ corresponding to the ${\mathbb G}_m$-bundle
 $\St_{2n,1} \ra St_{2n, 1}/ {\mathbb G}_m$, 
 \item
 The above line bundle on $\St_{2n, 1}/ {\mathbb G}_m$ pulls back under the map $\bar \phi_n$ (see \eqref{bphi.n}) to the line bundle on $\St_{2n,n}/({\mathbb G}_m \times \GL_{n-1})$
 corresponding to the ${\mathbb G}_m$-bundle $\St_{2n,n}/(1 \times \GL_{n-1} ) \ra \St_{2n, n}/({\mathbb G}_m \times \GL_{n-1})$.
\item The above line bundle on $\St_{2n, n}/({\mathbb G}_m \times \GL_{n-1} )$ pulls back to $\beta$ by the map 
 \newline \noindent
 $u_n:{\widetilde {\BN_{\GL_n}(\rmT_n)^{gm, n}}} = \St_{2n,n}/({\mathbb G}_m \times \rmN_{\GL_{n-1}}(T_{n-1}) ) \ra \St_{2n,n}/({\mathbb G}_m \times \GL_{n-1} )$: see ~\eqref{bar.p.n}).
 \item Recall that $\pi_{n} =\Sigma_{\T}^{\infty}\bar u_{n+}$. The above observations now show that the composite map 
 \[\lambda' \circ j_n \circ \pi_n= \lambda' \circ j_n \circ  \Sigma_{\T}^{\infty}\bar u_{n+} = \lambda' \circ j_n \circ \Sigma_{\T}^{\infty}u_{n+} \circ \Sigma_{\T}^{\infty}\bar \phi_{n+}: \Sigma_{\T}^{\infty}{\widetilde {\BN_{\GL_n}(\rmT_n)}}_+^{gm, n}{\overset {j_n\circ \pi_n} \longrightarrow} {\Sigma_{\T}^{\infty}\B{\mathbb G}_{m}^{\rm gm}} {\overset {\lambda'} \ra} f_1{\mathbf K} = \tilde {\mathbf K}\]
 corresponds to the bundle $\beta$.
\item Now $\lambda' \circ \zeta_n = \lambda' \circ j_n \circ  \pi_{n+} \circ \tau_n = \lambda' \circ j_n \circ  \Sigma_{\T}^{\infty}\bar u_{n+} \circ \tau_n = \lambda' \circ j_n \circ \Sigma_{\T}^{\infty}u_{n+} \circ \Sigma_{\T}^{\infty}\bar \phi_{n+} \circ \tau_n$: see ~\eqref{zeta.n}.
\end{enumerate}
Therefore,  it follows that $\lambda'_*(\zeta_{n*}(u))$ corresponds to the composite map:
\be \begin{equation}
\label{lambda.composite.map}
{\Sigma_{\T}^{\infty} {\BN_{\GL_n}(\rmT_{n})}_+ ^{\rm gm,n}} {\overset {id } \longrightarrow} {\Sigma_{\T}^{\infty} {\BN_{\GL_n}(\rmT_{n})}_+ ^{\rm gm,n}} {\overset {\tau_n} \longrightarrow} \Sigma_{\T}^{\infty}{\widetilde {\BN_{\GL_n}(\rmT_n)}}_+^{gm, n}{\overset {j_n\circ \pi_n} \longrightarrow} {\Sigma_{\T}^{\infty}\B{\mathbb G}_{m}^{\rm gm}} {\overset {\lambda'} \ra} f_1{\mathbf K} = \tilde {\mathbf K}.
\end{equation} \ee
Therefore, at this point, in order to prove the
commutativity of the bottom triangle in ~\eqref{key.comm.triangle.2}, it suffices to prove that 
\be \begin{equation}
\label{transf.tau.1}
\tau_n^*(\beta) = \alpha,
\end{equation} \ee
where $\tau_n^*$ denotes the transfer map induced by the transfer $\tau_n$ on the Grothendieck groups.
This is a straightforward computation making use of the
direct-images of coherent sheaves under finite \'etale maps as discussed in the following paragraphs, as well as Corollary ~\ref{tr.finite.etale.maps}. (See \cite[p. 142]{Beck} for 
very similar arguments in the topological case.)
 \vskip .2cm
Denoting the total space of the vector bundle $\alpha$ by $\rmE(\alpha)$, observe that  $\rmE(\alpha) = \St_{2n,n} {\underset {\rmN_{{\rm GL}_n}(\rmT_n)} \times} \rmW$, where $\rmW$ corresponds to the $n$-dimensional representation of ${\rm GL}_n$ forming the fibers of the vector bundle $\alpha$. 
 We will let $\rmW'$ denote the representation $\rmW$, but viewed as a 
representation of $\rmN_{{\rm GL}_n}(\rmT_n)$. Let $\rmT_{n-1}$ denote the $n-1$-dimensional split torus forming the last $n-1$-factors in
the split maximal torus $\rmT_n$. Observe that on further restricting to the action of the subgroup $\rmH=   {\mathbb G}_m \times \rmN_{{\rm GL}_{n-1}}(\rmT_{n-1})$, $\rmW'$ is the  representation of
$\rmN_{{\rm GL}_n}(\rmT_n)$ that is  induced from a $1$-dimensional representation $\rmV$ of the subgroup $\rmH$, that is,  if $\{\sigma_i\rmH |i=1, \cdots n\}$ is the complete set of left cosets of $\rmH$ in $\rmN_{\GL_n}(\rmT_n)$, then $\rmW' \cong \oplus_{i=1}^n \rmV_i$, with each $\rmV_i =\rmV$ and where $\rmN_{{\rm GL}_n}(\rmT_n)$ acts on $\rmW'$ as follows.
For $g \eps \rmN_{{\rm GL}_n}(\rmT_n)$, if 
 $g.\sigma_i = \sigma_k h$, $h \eps \rmH$, then $g.v_i =h_k$, with $v_i=v_k \eps \rmV$.  (This may be
 seen by observing that the normalizer of the maximal torus $\rmN_{{\rm GL}_n}(\rmT_n)$ is the semi-direct product of the symmetric group $\Sigma_n$ and
  the maximal torus $\rmT_n$.)
  \vskip .1cm
Next observe that $\rmp:{\widetilde {\BN_{\GL_n}(\rmT_n)^{gm, n}}} = \St_{2n,n} /\rmH \ra  \St_{2n,n}/\rmN_{\GL_n}(\rmT_n) = {\BN_{\GL_n}(\rmT_n)^{gm, n}}$ is a finite \'etale map of degree $n$. Then $\beta$ identifies with the 
line bundle, with structure group ${\mathbb G}_m$, defined by 
$ \St_{2n,n} {\underset {\rmH} \times}\rmV$ on ${\widetilde {\BN_{\GL_n}(\rmT_n)^{gm, n}}}= \St_{2n,n} /\rmH$. Clearly $\rmp_*(\beta) = \alpha$. Therefore, it suffices to show
that the transfer $\tau_n^*= \rmp_*$ in this case. That is, it suffices to prove that
the transfer $\tau_n^*$ on Grothendieck groups identifies with the push-forward in this case, which is proven in more generality in 
  the next section of this paper: see Corollary ~\ref{tr.finite.etale.maps}. This, therefore, completes the proof of the proposition.
  	\end{proof}
\vskip .2cm \noindent
{\bf Proof of Proposition ~\ref{key.comm.triangle.0}}. One observes from Proposition ~\ref{key.diagms}  that 
\[\bar q_{n*} = \lambda_* \circ \zeta_{n*},\]
and that the direct system of maps $\{\zeta_{n*}|n\}$ are compatible. Therefore, it follows that the maps $\{\bar q_{n*}|n\}$ are also compatible,
and taking the direct limit, we obtain $\bar q_*= \colimn \bar q_{n*} = \lambda_* \circ \colimn \zeta_{n*} = \lambda_* \circ \zeta_*$, which proves
 Proposition ~\ref{key.comm.triangle.0}. \qed
  	
\section{\bf The motivic transfer and the motivic Gysin maps associated to projective smooth morphisms}
\subsection{}
One may observe from the discussion in ~\eqref{transf.tau.0} and ~\eqref{transf.tau.1} that we need to define 
a transfer for all finite \'etale maps between smooth schemes with reasonable properties, such as compatibility with base-change,
and then show that such a transfer induces the pushforward map at the level of algebraic K-theory. The definition of
such a transfer map for finite \'etale maps is relatively straightforward. In fact there are existing constructions
in the literature, as the referee has pointed out, which either provide such transfers directly or can be used to provide such transfers with a little bit of
effort: we will in fact discuss some of these below in Remarks ~\ref{transf.finite.et.others}.
\vskip .1cm
However, proving that such transfers coincide with the pushforward map on Algebraic K-theory seems a bit involved: in the approach we take, one needs to first show
that these transfers coincide with Gysin maps for all orientable generalized motivic cohomology theories, and then observe that 
such Gysin maps on Algebraic K-theory agree with pushforward maps. A careful examination of the proof of the first statement will
 show that it takes more or less the same effort to define a transfer map for all {\it projective smooth} maps and show that 
 it agrees with a Gysin map up to multiplication by a certain Euler class, which will trivialize when the maps are finite \'etale.
 Therefore, we adopt this approach: observe that, as a consequence we are able to {\it derive the precise relationship between the 
 transfer and Gysin maps associated to all projective smooth maps between smooth quasi-projective schemes on all orientable generalized motivic cohomology theories}: we believe this result is of 
 independent interest, though not used in the rest of the paper in this generality.
	\vskip .1cm
	Therefore, the general context in which we work in this section will be
	the following. Let 
	\be \begin{equation}
	\label{basic.context.0}
	   \rmp: \rmE \ra \rmB
	\end{equation} \ee
	\vskip .1cm \noindent
	denote a {\it projective smooth map} of quasi-projective smooth schemes over the base field. 
	
	\subsection{The definition of a transfer for projective smooth maps}
	In order to motivate this construction, we will quickly review the corresponding {\it Thom-Pontrjagin construction} in the context of classical algebraic topology.  Here $\rmp: \rmE \ra \rmB$ will denote a smooth fiber bundle between compact manifolds $\rmE$ and $\rmB$. Then one may obtain a closed
	imbedding of $\rmE$ in $\rmB \times {\mathbb R}^N$ for ${\rm N}$ sufficiently large. We will denote this imbedding by $i$. Therefore, one obtains the Thom-Pontrjagin collapse map
	\be \begin{equation}
	\label{TP}
	{\rm {TP}}:  \rmB_+\wedge {\rm S}^N \ra {\rm Th(\nu)}
	\end{equation} \ee
	where $\nu$ denotes the normal bundle associated to the closed imbedding $i$. (One may recall that this is the starting point of the classical Atiyah duality (see \cite{At}, \cite{SpWh55}, \cite{DP}) as well as
	its \'etale variant as in \cite{J86} and \cite{J87} in the context of \'etale homotopy theory as in \cite{AM}.) 
	\vskip .1cm
	We proceed to define a corresponding construction in the motivic context, making use of {\it the Voevodsky collapse} in the
	place of the Thom-Pontrjagin collapse. In the situation in ~\eqref{basic.context.0}, as the schemes $\rmE$ and $\rmB$ are assumed to be
	quasi-projective, one obtains a closed immersion $i: \rmE \ra \rmB \times {\mathbb P}^N$ for a large enough $\rmN$. Therefore, the discussion
	in \cite[Proposition 2.7, Lemma 2.10 and Theorem 2.11]{Voev} (see also \cite[\S 10.4]{CJ20}) provides the Voevodsky collapse map
	\be \begin{equation}
	\label{V.collapse}
	 \rmV: \rmB_+ \wedge \T^n \ra {\rm {Th(\nu)}}
	\end{equation} \ee
	for a suitably large $n$, and where $\nu$ denotes the vector bundle on $\rmE$ we call the {\it virtual normal bundle}: see \cite[\S 10.8]{CJ20}. (See also \cite[5.3]{Ho} for a discussion on the collapse, which in this framework is originally due to Voevodsky.)
\vskip .2cm	
Let $\tau=\tau_{\rmE/\rmB}	$ denote the relative tangent bundle associated to $\rmp: \rmE \ra \rmB$. 	Assume the relative dimension of $\rmp$ is ${\rm d}$.
 Then it follows from \cite[Proposition 2.7 through Theorem 2.11]{Voev} (see also \cite[10.4, Definition 10.8]{CJ20}) that $\nu\oplus \tau $ is a trivial bundle on pull-back to $\widetilde \rmE$, where $\widetilde \rmE$ is a (functorial) affine
 replacement of $\rmE$ provided by the technique of Jouanolou (see \cite{Joun}).
 \begin{definition}
 Therefore, we may define the Becker-Gottlieb transfer
in the situation of ~\eqref{basic.context.0}  
as follows:
\be \begin{equation}
  \label{rel.transf}
  tr : \rmB_+ \wedge \T^n {\overset \rmV \ra} {\rm {Th(\nu)}} {\overset {i_{\nu}} \ra } Th(\nu\oplus \tau) \simeq \rmE_+ \wedge \T^n
\end{equation} \ee
\vskip .1cm \noindent
where $i_{\nu}$ is the map induced by the obvious inclusion $\nu \ra \nu \oplus \tau$. (See, for example, the proof of \cite[Theorem 4.3]{BG75}.)
\end{definition}
\vskip .2cm
\begin{proposition}
 \label{base.change}(Some basic properties of the transfer.) 
 \vskip .1cm
 (i) Assume that in ~\eqref{basic.context.0}, $\rmE = \rmE_0 \sqcup \rmE_1$. Denoting the
 corresponding transfers $tr_i: \rmB_+\wedge \T^n \ra \rmE_{i,+} \wedge \T^n$, and $tr: \rmB_+\wedge \T^n \ra \rmE_+ \wedge \T^n$,
 $tr^* = tr_0^*+tr_1^*$ in any generalized motivic cohomology theory.
 \vskip .1cm
 (ii) In case $\rmE= \rmB$ in ~\eqref{basic.context.0}, and the map $\rmp$ is the identity map on $\rmB$, then
 $tr^* = id$ on any orientable generalized motivic cohomology theory.
 \vskip .1cm
(iii) Assume that the square
 \begin{equation}
	\label{basic.context.1}
	   \xymatrix{{\rmE'} \ar@<1ex>[r] \ar@<1ex>[d]^{\rmp'} & {\rmE} \ar@<1ex>[d]^{\rmp}\\
	   	  {\rmB'} \ar@<1ex>[r] & {\rm B} }
	\end{equation} \ee
is cartesian.
Then we obtain the following homotopy commutative diagram of transfer maps:
\[\xymatrix{{\rmB_+ \wedge \T^n} \ar@<1ex>[r]^{\rmV} & {\rm {Th(\nu)}} \ar@<1ex>[r]^(.3){i_{\nu}} & {\rm Th(\nu\oplus \tau) \simeq \rmE_+ \wedge \T^n}\\
            {\rmB'_+ \wedge \T^n} \ar@<1ex>[r]^{\rmV'} \ar@<1ex>[u] & {\rm {Th(\nu')}} \ar@<1ex>[r]^(.3){i'_{\nu}} \ar@<1ex>[u]& {Th(\nu'\oplus \tau') \simeq \rmE'_+ \wedge \T^n,} \ar@<1ex>[u]}
\]
where $\rmV'$ and $i'_{\nu}$ are the maps corresponding to $\rmV$ and $i_{\nu}$ when $\rmB$ and $\rmE$ are replaced by $\rmB'$ and $\rmE'$.
\end{proposition}
\begin{proof} The proofs of the first and last statements are straightforward from the construction of the transfer.
The second statement follows Corollary ~\ref{tr.finite.etale.maps} by taking the map $\rmp$ to be the identity.
\end{proof}

\begin{remarks}
 \label{transf.finite.et.others}
 Here we briefly discuss other possible constructions of the transfer associated to {\it finite \'etale maps} $\rmp: \rmE \ra \rmB$, where $\rmE$ and $\rmB$ are quasi-projective smooth schemes over the base field $k$. One may find
 one such construction in \cite[2.3]{RO}, as pointed out by the referee. It is verified in {\it op.cit} that this transfer is compatible with base-change.
 (Making use of the 6-functor formalism in motivic homotopy theory as in \cite{Ay}, the second author has also sketched a construction
 of a transfer for finite \'etale maps. As this is not all that different from the one in \cite[2.3]{RO}, we do not discuss this any further here.)
 Therefore, what one needs to show is that this transfer on Algebraic K-theory agrees with the pushforward. As this is not 
 discussed in \cite{RO}, all one can say is that a proof of this fact will likely follow the same steps as outlined above and discussed below in detail,except that 
 the relative tangent bundle to the map $\rmp$ will be trivial in this case.
 \vskip .1cm
 The referee has also pointed out that the discussions in \cite{EHK+} and \cite{HJN+}, making use of framed correspondences, provide
 a transfer map for finite \'etale maps that identify with pushforwards (of vector bundles) on Algebraic K-theory.
\end{remarks}
\vskip .2cm
In the rest of this section, we do the following:
\begin{enumerate}[\rm(i)]
\item Making use of the same Voevodsky-collapse used in the construction of the transfer, we proceed to define a Gysin map
associated to projective and smooth maps between smooth quasi-projective schemes in all {\it orientable} generalized motivic cohomology theories.
\item Then we show that this Gysin map agrees with the Gysin maps defined by more traditional means, typically by factoring the  given 
map $\rmp: \rmE \ra \rmB$  as the composition of a closed immersion of $\rmE$ into a relative projective space $\rmB \times {\mathbb P}^n$ followed
 by the projection $\pi: \rmB \times {\mathbb P}^n \ra \rmB$. 
\item At this point, standard comparison results (see \cite[2.9.1]{P09} which invokes \cite[3.16, 3,17 and 3.18]{T-T}) show that the above Gysin maps identify with the push-forward maps on the Algebraic K-theory 
of smooth quasi-projective schemes.
\item Finally, we show that on orientable generalized motivic cohomology theories, the map induced by the transfer and the Gysin maps constructed below differ only by multiplication by
 the Euler class of the relative tangent bundle to the map $\rmp$. As a result, when $\rmp$ is a finite \'etale map, the relative tangent bundle to the map $\rmp$ trivializes and the map 
 induced by the transfer agrees with the push-forward on Algebraic K-theory.
\end{enumerate}

	\vskip .2cm
	\subsection{Gysin maps associated to projective smooth maps on orientable generalized motivic cohomology theories}
	
	\vskip .2cm
	We begin by quickly reviewing the corresponding situation in Algebraic Topology.  
	For any generalized cohomology theory $\rmh^*$, the Thom-Pontrjagin collapse in ~\eqref{TP} induces the 
	map 
	\[{\rm TP}^*: \rmh^*({\rm Th(\nu)}) \ra \rmh^*(\rmB_+ \wedge {\rm S}^N).\]
	We will further assume that $\rmh^*$ is an orientable cohomology theory in the sense that it has 
	a Thom-class $\rmT(\nu) \eps \rmh^{c}({\rm {Th(\nu)}})$ (where $c$ is the codimension of $\rmE$ in $\rmB \times \rmS^N$), so that cup product with this class defines 
	the Thom-isomorphism: $\rmh^*(\rmE) \ra \rmh^{*+c}({\rm {Th(\nu)}})$. Moreover, in this case one also observes the suspension isomorphism: 
	$\rmh^*(\rmB_+ \wedge \rmS^N ) \cong \rmh ^{*-N}(\rmB)$. Thus the composition 
	\be \begin{equation}
	\label{Gysin.TP.coll}
	\rmp_*: \rmh^*(\rmE) {\overset {\cup \rmT(\nu) } \ra } \rmh^{*+c}({\rm {Th(\nu)}}) {\overset {{\rm {TP}}^*} \ra} \rmh^{*+c}(\rmB_+ \wedge \rmS^N) \cong \rmh^{*+c-N}(\rmB)
	\end{equation} \ee
	defines a Gysin map.  One may observe that if the relative dimension of $\rmE$ over $\rmB$ is $d$, then ${\rm c=N-d}$, so that $\rmh^{*+c-N}(\rmB) = \rmh^{*-d}(\rmB)$ as required of a Gysin map. 
	\vskip .2cm
	We proceed to define a corresponding Gysin map in the motivic context, {\it for orientable generalized motivic cohomology theories} in the sense of  \cite[\S 2]{PY02} (see also \cite{P09}), making use of the Voevodsky collapse in the
	place of the Thom-Pontrjagin collapse. In the situation in ~\eqref{basic.context.0}, as the schemes $\rmE$ and $\rmB$ are assumed to be
	quasi-projective, one obtains a closed immersion $i: \rmE \ra \rmB \times {\mathbb P}^N$ for a large enough $\rmN$. 
	In this context, we recall the Voevodsky collapse
	\be \begin{equation}
	\label{V.collapse.2}
	 \rmV: \rmB_+ \wedge \T^n \ra {\rm {Th(\nu)}}
	\end{equation} \ee
	as discussed above in ~\eqref{V.collapse}.
	It should be clear that, with this collapse map replacing the Thom-Pontrjagin collapse, and generalized
	motivic cohomology theories that are orientable (and bi-graded), one obtains a {\it Gysin map}
	\be \begin{equation}
	\label{Gysin.V.coll}
	\rmp_*: \rmh^{*, \bullet}(\rmE) {\overset {\cup \rmT(\nu) } \ra } \rmh^{*+2c, \bullet +c}({\rm {Th(\nu)}}) {\overset {\rmV^*} \ra} h^{*+2c, \bullet +c}(\rmB_+ \wedge \T^n) \cong \rmh^{*+2c-2n, \bullet +c-n}(\rmB) = \rmh^{*-2d, \bullet -d}(\rmB),
	\end{equation} \ee
	\vskip .1cm \noindent
	if ${\rm d}$ is the relative dimension of $\rmE$ over $\rmB$, $\rmT(\nu)$ is the Thom-class of the bundle $\nu$, and ${\rm c}$ is the rank of the vector bundle $\nu$.
	\vskip .2cm 
	Next we proceed to show that the Gysin map defined above indeed agrees with Gysin maps that are defined by other more traditional means,
	such as in \cite{P09} or \cite[\S 4 and \S 5]{PY02}. (See also \cite{Deg}.) For this,
	we need to first recall the framework for defining the Voevodsky collapse.
	One may observe from \cite[pp. 69-70]{Voev} that one needs to consider the sequence of 
	closed immersions 
	\be \begin{equation}
	\label{Voev.setup}
	\rmE {\overset i \ra} \rmB\times {\mathbb P}^d {\overset {id \times \Delta} \longrightarrow} \rmB \times {\mathbb P}^d \times {\mathbb P}^d {\overset {id \times Segre} \longrightarrow}\rmB \times {\mathbb P}^{d^2+2d},
	\end{equation} \ee
	where {\rm Segre} denotes the Segre imbedding. {\it We will let ${\rm m= d^2+2d}$ henceforth}. Let $\nu$
	denote the normal bundle to the above composite closed immersion and let $c$ denote
	the codimension of this closed immersion. Then one obtains the
	following sequence of maps:
	\be \begin{multline}
	\begin{split}
	\label{Gysin.2}
	\rmh^{*, \bullet}(\rmE) {\overset {\cup \rmT(\nu)} \longrightarrow} \rmh^{*+2c, \bullet +c}(\Th(\nu)) \cong \rmh^{*+2c, \bullet +c}(\rmB \times {\mathbb P}^{m}/ (\rmB \times {\mathbb P}^{m} - \rmE)) {\overset {Gysin_1'} \ra } \rmh^{*+2c, \bullet +c}(\rmB \times {\mathbb P}^{m})\\
	 {\overset {Gysin_2} \ra } \rmh^{*+2c-2m, \bullet +c-m}(\rmB). 
	 \end{split}
	\end{multline} \ee
	Here the map denoted $Gysin_1'$ precomposed with the cup product with the Thom class $\rmT(\nu)$  is the usual Gysin map associated to the composite closed immersion $\rmE \ra \rmB \times {\mathbb P}^{m}$ (see \cite[\S 4]{PY02}) and the map denoted $Gysin_2$ is 
	the usual Gysin map associated to the projection $\rmB \times {\mathbb P}^{m} \ra \rmB$: see \cite[Definition 5.1]{PY02}. 
	\vskip .2cm
	\subsubsection{Deformation to the normal cone}
	\label{deform.norm.cone}
	In order to relate the Gysin maps in ~\eqref{Gysin.V.coll} and ~\eqref{Gysin.2}, we first invoke the technique of deformation to the normal cone
	from \cite[1.2.1, Theorem 1.2]{P09}. (See also \cite[2.2.8 Theorem]{P}). Let $i: \rmY \ra \rmX$ denote a closed immersion of smooth schemes of finite type over $k$ with normal bundle ${\mathcal N}$.
	Then there exists a smooth scheme $\tilde \rmX$ together with a smooth map $p: \tilde \rmX \ra {\mathbb A}^1$ and a closed immersion $i: \rmY \times {\mathbb A}^1 \ra \tilde \rmX$, so that
	the composition $p \circ i: \rmY \times {\mathbb A}^1 \ra {\mathbb A}^1$ coincides with the projection $\rmY \times {\mathbb A}^1 \ra {\mathbb A}^1$. Moreover the following additional properties hold:
	\begin{enumerate}
	 \item The fiber of $p$ over $1 \in {\mathbb A}^1$ is isomorphic to $\rmX$ and the base change of $i$ by the imbedding $1 \in {\mathbb A}^1$
	 corresponds to the given imbedding $\rmY \ra \rmX$.
	 \item The fiber of $p$ over $0 \in {\mathbb A}^1$ is isomorphic to ${\mathcal N}$ and the base change of $i$ by the imbedding $0 \in {\mathbb A}^1$
	 corresponds to the $0$-section imbedding $\rmY \ra {\mathcal N}$.
	 \item If $\rmZ \ra \rmY$ is a closed immersion of a smooth subscheme of $\rmY$, then one obtains the following diagram
	 \[({\mathcal N}, {\mathcal N} -\rmZ) {\overset {i_0} \ra} (\tilde \rmX, \tilde \rmX - \rmZ \times {\mathbb A}^1) {\overset {i_1} \leftarrow} (\rmX, \rmX - \rmZ),\]
	 and hence the following diagram 
	 for any orientable generalized motivic
	 cohomology theory $\rmh^{*, \bullet}$ with both the horizontal maps being isomorphisms: \footnote{This may be viewed as a cohomology-variant of the purity Theorem: see \cite[Theorem 2.23, p. 115]{MV}.}
	 \be \begin{equation}
	   \label{deform.norm.cone.1}
	  \rmh^{*, \bullet}({\mathcal N}, {\mathcal N} - \rmZ) {\overset {i_0^*} \leftarrow} \rmh^{*, \bullet}( \tilde \rmX, \tilde \rmX - \rmZ \times {\mathbb A}^1) {\overset {i_1^*} \ra} \rmh^{*, \bullet}(\rmX, \rmX - \rmZ).
	 \end{equation} \ee
        \item Moreover, in the above situation the normal bundle to the composite closed immersion $\rmZ \ra \rmY \ra {\mathcal N}$ is isomorphic
        to the sum ${\mathcal N}_{\rmZ, \rmY} \oplus {\mathcal N}_{|\rmZ}$, where ${\mathcal N}_{\rmZ, \rmY}$ denotes the normal bundle associated to the closed immersion $\rmZ \ra \rmY$.
	\end{enumerate}

	\begin{proposition} 
		\label{Gysin.agree}
		Assume the above situation. Then the Gysin map defined in ~\eqref{Gysin.V.coll} agrees with the Gysin map defined in ~\eqref{Gysin.2}.
	\end{proposition}
	\begin{proof}
	
	We will follow the constructions in \cite[pp. 69-70]{Voev}. Accordingly the second
	${\mathbb P}^{\rmd}$ in ${\mathbb P}^{\rmd} \times {\mathbb P}^{\rmd}$ in ~\eqref{Voev.setup} is identified with the dual
	projective space and $\bH$ will denote the incidence hyperplane in ${\mathbb P}^{\rmd} \times {\mathbb P}^{\rmd}$. Then it is observed there that, ${\widetilde {\mathbb P}^d} ={\mathbb P}^d \times {\mathbb P}^d - \bH$, considered as a scheme over ${\mathbb P}^d$ by the
	projection to the first factor ($\rmp_1$) is an affine space bundle: in fact, this is an instance
	of what is known as Jouanolou's trick. 
	\vskip .1cm
	Let $\rmN$ denote the normal bundle to the 
	Segre imbedding of $\rmB \times {\mathbb P}^{\rmd} \times {\mathbb P}^{\rmd}$ in $\rmB \times {\mathbb P}^{\rm n}$. If $j: \rmB \times {\widetilde {\mathbb P}^d} \ra \rmB\times {\mathbb P}^d \times {\mathbb P}^d$ denotes the open immersion, we let $j^*(\rmN)$ be the pull-back of $\rmN$ to $\rmB \times {\widetilde {\mathbb P}^d}$. Since 
	${\widetilde {\mathbb P}^{\rmd}} \ra {\mathbb P}^{\rmd}$ is a torsor for a vector bundle,
	and all vector bundles on affine spaces are trivial by the affirmative solution to the Serre conjecture, one may see that $j^*(\rmN) \cong \rmp_1^*(\Delta^*(\rmN))$, where
	$\Delta: \rmB \times {\mathbb P}^{\rmd} \ra \rmB \times {\mathbb P}^{\rmd} \times {\mathbb P}^{\rmd}$ is the diagonal imbedding.
	\vskip .1cm
	Let $\cE$ denote the pull-back $\rmp_1^*(\tau_{{\rmB\times{\mathbb P}^d}/\rmB})$, (where $\tau_{{\rmB \times {\mathbb P}^d}/\rmB}$ denotes the relative tangent bundle to $\rmB\times {\mathbb P}^d$ over $\rmB$), and let  $\nu_1$ denote the normal bundle to the closed immersion
	$i:\rmE \ra \rmB \times {\mathbb P}^d$. Let $\rmp:  \tilde \rmE \ra \rmE$ denote the map induced by $\rmp_1: \rmB \times
	{\widetilde {\mathbb P}}^d \ra \rmB\times {\mathbb P}^d$ when $\tilde \rmE$ is defined by
	the cartesian square:
	\[\xymatrix{{\tilde \rmE} \ar@<1ex>[r] \ar@<1ex>[d]^{\rmp}& {\rmB \times {\widetilde {\mathbb P}}^d } \ar@<1ex>[d]^{\rmp_1}\\
		        {\rmE} \ar@<1ex>[r] & {\rmB \times {\mathbb P}^d}.} \]
	
	Let $\tilde \nu_1$ denote the normal bundle to the induced closed immersion $\tilde \rmE \ra \rmB \times {\widetilde {\mathbb P}}^d$.
	Then denoting the Thom-class of the bundle $\cE_1=(\cE \oplus j^*(\rmN))_{|\tilde \rmE} \oplus \tilde \nu_1$ by $\rmT(\cE_1)$, we obtain the sequence of maps:
	\be \begin{equation}
	\label{Gysin.3}
	   \rmh^{*, \bullet}(\tilde \rmE) {\overset {\cup \rmT(\cE_1)} \longrightarrow } \rmh^{*+2c, \bullet +c}(\Th(\cE_1)) \ra  \rmh^{*+2c, \bullet +c}(\Th(\cE \oplus j^*(\rmN ))). 
	\end{equation} \ee
	The last map is obtained from the observation that the normal bundle to the composite closed immersion $\tilde \rmE \ra {\rmB \times {\widetilde {\mathbb P}}^d } {\overset {0-section} \longrightarrow} \cE \oplus j^*(\rmN)$ is $\cE_1$: see 
	~\ref{deform.norm.cone}(4) above.
	At this point we make use of the identification $j^*(\rmN) = \rmp_1^*(\Delta^*(\rmN))$, so that $\cE_1 = \rmp^*(i^*(\tau_{\rmB \times {\mathbb P}^d/\rmB} \oplus \Delta^*(\rmN)) \oplus \nu_1)$, and $\cE \oplus \rmN= \rmp_1^*(\tau_{\rmB \times {\mathbb P}^d/\rmB} \oplus \Delta^*(\rmN))$, which then readily provides the commutativity of the following diagram
	 with $\cE_0 = i^*(\tau_{\rmB \times {\mathbb P}^d/\rmB} \oplus \Delta^*(\rmN)) \oplus \nu_1$:
	\vskip .1cm
	\be \begin{equation}
	\label{Gysin.4}
	\xymatrix{{\rmh^{*, \bullet}(\tilde \rmE)} \ar@<1ex>[r]^{\cup \rmT(\cE_1)} \ar@<1ex>[d]^{\cong} & {\rmh^{*+2c, \bullet +c}(\Th(\cE_1))} \ar@<1ex>[r] \ar@<1ex>[d]^{\cong}  & {\rmh^{*+2c, \bullet +c}(\Th(\cE \oplus j^*(\rmN ) ))} \ar@<1ex>[d]^{\cong}\\
	{\rmh^{*, \bullet}(\rmE)} \ar@<1ex>[r]^(.3){\cup \rmT(\cE_0)} & {\rmh^{*+2c, \bullet+c}(\rmB\times {\mathbb P}^{m}/(\rmB \times {\mathbb P}^{m}-E ))} \ar@<1ex>[r] & {\rmh^{*+2c, \bullet+c}(\rmB\times {\mathbb P}^{m}/(\rmB \times {\mathbb P}^{m}-(\rmB \times \Delta {\mathbb P}^d ))) }  .} 
	\end{equation} \ee
	\vskip .1cm
	The left-most vertical map above is an isomorphism as $\rmp:\tilde \rmE \ra \rmE$ is an affine space bundle. To see that the 
	next
	 vertical map is an isomorphism, one needs to observe that 
	\[{\rm Th}(\cE_1) = {\rm Th}(\rmp^*(i^*(\tau_{\rmB \times {\mathbb P}^d/\rmB} \oplus \Delta^*(\rmN)) \oplus \nu_1))  \simeq {\rm Th}(i^*(\tau_{\rmB \times {\mathbb P}^d/\rmB} \oplus \Delta^*(\rmN)) \oplus \nu_1).\]
	At this point we make use of the isomorphisms in ~\eqref{deform.norm.cone.1} and ~\eqref{deform.norm.cone}(4) to obtain the
	 isomorphism:
	\[\rmh^{*, \bullet}({\rm Th}(i^*(\tau_{\rmB \times {\mathbb P}^d/\rmB} \oplus \Delta^*(\rmN)) \oplus \nu_1)) \cong \rmh^{*, \bullet}( (\rmB \times {\mathbb P}^m)/(\rmB \times {\mathbb P}^m - \rmE)).\]
	To see the last vertical map in ~\eqref{Gysin.4} is an isomorphism, one first observes that
	\[{\rm Th}( \cE \oplus j^*(\rmN)) = {\rm Th}(\rmp_1^*(\tau_{{\rmB\times{\mathbb P}^d}/\rmB} \oplus \Delta ^*({\rm N}))) \simeq {\rm Th} (\tau_{{\rmB\times{\mathbb P}^d}/\rmB} \oplus \Delta ^*({\rm N})),\]
        and then adopts a similar argument to obtain the isomorphism:
       \[\rmh^{*, \bullet}({\rm Th} (\tau_{{\rmB\times{\mathbb P}^d}/\rmB} \oplus \Delta ^*({\rm N})) \cong  \rmh^{*, \bullet}((\rmB \times {\mathbb P}^m)/ ((\rmB \times {\mathbb P}^m) - (\rmB \times \Delta {\mathbb P}^d)).\]
	 \vskip .1cm
	We also obtain the following commutative diagram:
	\small
	\vskip .1cm
	\be \begin{equation}
	\label{Gysin.5}
	\xymatrix{{\rmh^{*+2c, \bullet+c}(\rmB\times {\mathbb P}^{m}/((\rmB \times {\mathbb P}^{m}- \rmB \times ({\mathbb P}^d \times {\mathbb P}^d)) \cup \rmB \times \bH_{\infty})) }  \ar@<1ex>[r] \ar@<1ex>[d] & {\rmh^{*+2c, \bullet+c}(\rmB\times {\mathbb P}^{m}/( \rmB \times \bH_{\infty})) } \ar@<1ex>[r]^{\cong} \ar@<1ex>[d]^{} & {\rmh^{*+2c, \bullet+c}(\rmB _+ \wedge \T ^{m})} \ar@<1ex>[d]^{\cong}\\
		{\rmh^{*+2c, \bullet+c}(\rmB\times {\mathbb P}^{m}/(\rmB \times {\mathbb P}^{m}-(\rmB \times {\mathbb P}^d \times{\mathbb P}^d ))) } \ar@<1ex>[r]^(.6){} & 	{\rmh^{*+2c, \bullet+c}(\rmB\times {\mathbb P}^{m})} \ar@<1ex>[r]^{Gysin} & {\rmh^{*+2c-2m, \bullet+c-m}(\rmB)},}
	\end{equation}
	\vskip .2cm \noindent
	\normalsize
	where $\bH_{\infty}$ is a hyperplane in $\rmB \times {\mathbb P}^m$ which pulls back to the incidence hyperplane in $\rmB \times {\mathbb P}^d \times {\mathbb P}^d$ under
	the Segre imbedding. 
	Next we recall the following identification 
	(see \cite[proof of Lemma 2.10]{Voev}):
	\[\Th(\cE \oplus j^*(\rmN)) \simeq \rmB\times {\mathbb P}^{m}/((\rmB \times {\mathbb P}^{m}- \rmB \times ({\mathbb P}^d \times {\mathbb P}^d)) \cup \rmB \times \bH_{\infty}).\]
	This shows that in this case, one obtains a composite {\it collapse map}
	\be \begin{equation}
	 \label{V.collapse.Segre}
	 {\rm V}: \rmB_+ \wedge \T^m \ra \Th(\cE \oplus j^*(\rmN)) \ra \Th(\cE_1),
	\end{equation} \ee
        and hence that one may compose the maps forming the top row of the diagram ~\eqref{Gysin.4} followed
	by the maps forming the top row of the  diagram ~\eqref{Gysin.5}. In view of the fact that $\rmp: \tilde \rmE \ra \rmE$
	is an affine replacement, $\Th(\cE_1) \simeq \Th(\cE_0)$, so that the collapse map in ~\eqref{V.collapse.Segre} defines
	a collapse ${\rmV}:  \rmB_+ \wedge \T^m \ra \Th(\cE_0)$, which differs from the collapse map in ~\eqref{V.collapse} only by 
	the addition of a trivial bundle, and hence a $\T$-suspension of some finite degree on both the source and the target: see \cite[proof of Proposition 2.7 and Theorem 2.11]{Voev}. Therefore, 
	one may now observe that the composition of the maps in the top rows of the two diagrams followed by the suspension isomorphism forming the right most vertical map in the second diagram identifies with the map $\rmp_*$ in ~\eqref{Gysin.V.coll}.
	\vskip .1cm
	Observe that there
	is a natural map:
	\[  \rmh^{*+2c, \bullet+c}(\rmB\times {\mathbb P}^{m}/(\rmB \times {\mathbb P}^{m}-(\rmB \times \Delta {\mathbb P}^d ))) \ra \rmh^{*+2c, \bullet+c}(\rmB\times {\mathbb P}^{m}/(\rmB \times {\mathbb P}^{m}-(\rmB \times {\mathbb P}^d \times{\mathbb P}^d )) .\]
Therefore one may compose the maps forming the bottom rows of the two diagrams ~\eqref{Gysin.4}] and ~\eqref{Gysin.5}.
	The composition of the maps forming the bottom rows of the two diagrams defines the Gysin map in ~\eqref{Gysin.2}. The commutativity of the two diagrams proves these two maps are the same.
	\end{proof}

\begin{theorem}
	\label{compat.Gysin}
	Let $\rmh^{*, \bullet}$ denote a generalized motivic cohomology which is orientable in the above sense. Let $tr$ denote the transfer as in ~\eqref{rel.transf}. Then if $eu(\tau)$ denotes the Euler class
	of the bundle $\tau$, we obtain the relation:
	\be \begin{equation}
	   tr^*(\alpha) = \rmp_*(\alpha \cup eu(\tau)), \alpha \eps  \rmh^{*, \bullet}(\rmE)
	\end{equation} \ee
	\vskip .1cm \noindent
	where $\rmp_*$ denotes the Gysin map defined above in ~\eqref{Gysin.V.coll}. 
\end{theorem}
\begin{proof}
	As shown in \cite[Theorem 4.3]{BG75}, and adopting the terminology as in ~\eqref{V.collapse} and ~\eqref{Gysin.V.coll},  it suffices to prove the commutativity of the following diagram:
	\be \begin{equation}
	   \label{comp.Gysin.1}
	    \xymatrix{{\rmh^{*, \bullet}(\rmE)}  \ar@<1ex>[rr]^{\cup eu(\tau)} \ar@<1ex>[d]^{\cong} &&{\rmh^{*+2d, \bullet+d  }(\rmE)} \ar@<1ex>[r]^{\rmp_*} \ar@<1ex>[d]^{\cup {\rm {T(\nu)}}}  & {\rmh^{*, \bullet} (\rmB) } \ar@<1ex>[d]^{\cong}\\
	    	      {\rmh^{*+2n, \bullet+n}(\rmE_+\wedge \T^n)}   \ar@<1ex>[r]^{\cong}  &  {\rmh^{*+2n, \bullet+n}(\rm {Th(\nu \oplus \tau)})} \ar@<1ex>[r]^{i^*}  & {\rmh^{*+2n, \bullet +n} ({\rm {Th(\nu)}}) } \ar@<1ex>[r]^{\rmV^*} & {\rmh^{*+2n, \bullet +n} (\rmB_+ \wedge \T^n) }.} 
	    \end{equation} \ee
	\vskip .1cm \noindent
	Here $d$ is the relative dimension of $\rmE$ over $\rmB$ and $i$ denotes the map of Thom-spaces induced by the inclusion $\nu \ra \nu \oplus \tau$. The definition of the Gysin map as in ~\eqref{Gysin.V.coll} readily proves the commutativity of the right square, so that it  suffices to prove the commutativity of the left square. This results from the commutativity of the diagram:
	\be \begin{equation}
	\label{comp.Gysin.2}
	\xymatrix{ {\rmh^{*, \bullet}(\rmE)}  \ar@<1ex>[r]^{\cup eu(\tau)}  \ar@<1ex>[d]^{\cup \rmT(\tau)} & {\rmh^{*+2d, \bullet+d  }(\rmE)} \ar@<1ex>[d]^{\cong} \ar@<1ex>[r]^{\cup \rmT(\nu)} & {\rmh^{*+2n, \bullet+n  }({\rm {Th(\nu)}})}  \ar@<1ex>[d]^{\cong}\\
		{\rmh^{*+2d, \bullet+d}({\rm {Th(\tau)}})}  \ar@<1ex>[r] \ar@<1ex>[d]^{\cup \rmT(\nu)}   & {\rmh^{*,+2d \bullet+d}(\rmE(\tau))}  \ar@<1ex>[r]^{\cup \rmT(\pi_1^*(\nu))} 
	             \ar@<1ex>[d]^{\cup \rmT(\pi_1^*(\nu))} & {\rmh^{*+2n, \bullet+n  }({\rm {Th(\pi_1^*(\nu))}})} \ar@<1ex>[dl]^{id}\\
	    {\rmh^{*+2n, \bullet+n}({\rm {Th(\tau \oplus \nu)}})  }  \ar@<1ex>[r] & {\rmh^{*+2n, \bullet+n}({\rm {Th(\pi_1^*(\nu))}} ) } .}
	\end{equation} \ee
	\vskip .1cm \noindent
	Here, if $\alpha$ denotes a vector bundle, ${\rm {Th(\alpha)}}$ ($\rmT(\alpha)$) denotes the Thom space (The Thom-class, \res) of $\alpha$.
	Observe that the composition of the top row and the right vertical map in the left square of ~\eqref{comp.Gysin.1} equals the
	composition of the maps in the top row of ~\eqref{comp.Gysin.2}. The composition of the map in the left column and the first bottom map in ~\eqref{comp.Gysin.1} clearly equals the  composition of the two vertical maps in the left most column of ~\eqref{comp.Gysin.2}.
	Since $\rmE(\tau)$ denotes the total space of the vector bundle $\tau$, we obtain the isomorphism ${\rmh^{*+2d, \bullet+d  }(\rmE)} {\overset {\cong} \ra} {\rmh^{*+2d, \bullet+d}(\rmE(\tau))} $ and also the isomorphism ${\rmh^{*+2n, \bullet +n} ({\rm {Th(\nu)}}) } {\overset {\cong} \ra} {\rmh^{*+2n, \bullet+n  }({\rm {Th(\pi_1^*(\nu))}} )}$, where $\pi_1: {\rm E}(\tau) \ra {\rm E}$ denotes the projection. Moreover, under the above isomorphisms, the map denoted $i^*$ in ~\eqref{comp.Gysin.1} identifies with the bottom most map in
	~\eqref{comp.Gysin.2}. These observations prove the commutativity of the diagram ~\eqref{comp.Gysin.2} and hence the commutativity of the diagram ~\eqref{comp.Gysin.1} as well. 
	\end{proof}
\begin{corollary}
	\label{tr.finite.etale.maps}
	Let $\rmp: \rmE \ra \rmB$ denote a finite \'etale map between smooth quasi-projective schemes. If $\rmh^{*, \bullet}$ is an orientable generalized motivic cohomology theory defined by a motivic spectrum, then one has the equality:
	\[tr^* = \rmp_*\]
	where $tr^*$ denotes the map induced by the motivic Becker-Gottlieb transfer $tr$ \rm(see ~\eqref{rel.transf} \rm) {\it in the above cohomology theory and $\rmp_*$ denotes the Gysin map. Moreover, for Algebraic K-Theory, the Gysin map $\rmp_*$ agrees with 
	the finite pushforward defined for coherent sheaves.}
\end{corollary}	
\begin{proof} The first statement is an immediate consequence of Theorem ~\ref{compat.Gysin}, once one observes that the Euler class $eu(\tau)$ is trivial, which follows from the fact that $\rmp$ is finite \'etale and $\tau$ denotes the relative tangent bundle of the map $\rmp$. 
The second statement on the Gysin map $\rmp_*$ for Algebraic K-Theory follows from \cite[2.9.1]{P09}, invoking \cite[3.16, 3,17 and 3.18]{T-T}. Observe that pushforward by 
finite \'etale maps send vector bundles to vector bundles, and for smooth quasi-projective schemes over $k$, the K-theory of coherent sheaves identifies with the 
K-theory of vector bundles. 
\end{proof}



\end{document}